\documentclass[11pt,reqno]{amsart}
\usepackage{graphicx, enumerate, url}
\usepackage{amsmath,amssymb,mathtools}
\usepackage{color}
\usepackage{tikz}

\usepackage[margin=1in]{geometry}

\let\O\undefined
\DeclareMathOperator{\O}{O}

\DeclareMathOperator{\tr}{tr}
\DeclareMathOperator{\Diag}{diag}
\DeclareMathOperator{\Rank}{rank}

\DeclareMathOperator{\V}{V}

\DeclareMathOperator{\Vol}{Vol}

\DeclareMathOperator{\im}{im}
\DeclareMathOperator{\Gr}{Gr}
\DeclareMathOperator{\Graff}{Graff}

\DeclareMathOperator{\Flag}{Flag}

\theoremstyle{plain}
\newtheorem{theorem}{Theorem}%[section]
\newtheorem{proposition}[theorem]{Proposition}
\newtheorem{lemma}[theorem]{Lemma}
\newtheorem{corollary}[theorem]{Corollary}

\theoremstyle{definition}
\newtheorem{definition}[theorem]{Definition}

\theoremstyle{remark}

\begin{document}

\title[Distances between subspaces of different dimensions]{Schubert varieties and distances between subspaces of different dimensions}
\author[K.~Ye]{Ke~Ye}
\address{Department of Mathematics, University of Chicago, Chicago, IL 60637.}
\email{kye@math.uchicago.edu}
\author[L.-H.~Lim]{Lek-Heng~Lim}
\address{Computational and Applied Mathematics Initiative, Department of Statistics,
University of Chicago, Chicago, IL 60637.}
\email[corresponding author]{lekheng@galton.uchicago.edu}

\begin{abstract}
We resolve a basic problem on subspace distances that often arises in applications: How can the usual Grassmann distance between equidimensional subspaces be extended to subspaces of different dimensions? We show that a natural solution is given by the distance of a point to a Schubert variety within the Grassmannian. This distance reduces to the Grassmann distance when the subspaces are equidimensional and  does not depend on any embedding into a larger ambient space. Furthermore, it has a concrete expression involving principal angles, and is efficiently computable in numerically stable ways. Our results are largely independent of the Grassmann distance --- if desired, it may be substituted by any other common distances between subspaces. Our approach depends on a concrete algebraic geometric view of the Grassmannian that parallels the differential geometric perspective that is well-established in applied and computational mathematics.
\end{abstract}



\maketitle

%%%%
\section{Introduction}\label{sec:motivation}

%%%%
%\subsection{Motivations}
Biological data (e.g.\ gene expression levels, metabolomic profile), image data (e.g.\  \textsc{mri} tractographs, movie clips), text data (e.g.\ blogs, tweets), etc., often come in the form of  a set of feature vectors $a_1,\dots,a_m \in \mathbb{R}^d$ and can be conveniently represented by a matrix $A \in \mathbb{R}^{m \times d}$ (e.g.\ gene-microarray matrices  of gene expression levels, frame-pixel matrices of grey scale values, term-document matrices of term frequencies-inverse document frequencies). In modern applications, it is often the case that one will encounter an exceedingly large sample size $m$ (massive) or  an exceedingly large number of variables $d$ (high-dimensional) or both.

The raw data $A$ is usually less interesting and informative than the spaces it defines, e.g.\ its row and column spaces or its principal subspaces. Moreoever, it often happens that  $A$ can be  well-approximated by a subspace $\mathbf{A} \in \Gr(k,n)$ where $k \ll m$ and $n \ll d$. The process of getting from $A$ to $\mathbf{A}$ is well-studied, e.g.\ randomly sample a subset of representative landmarks or compute principal components. 

Subspace-valued data appears in a wide range of applications: computer vision \cite{MYDF, VidMaSas2005}, bioinformatics \cite{HH}, machine learning \cite{HRS, LZ}, communication \cite{LHS,ZT}, coding theory \cite{AC,BN,CHS,DHST}, statistical classification \cite{HamLee2008}, and system identification \cite{MYDF}. In computational mathematics, subspaces arise in the form of Krylov subspaces \cite{LS} and their variants \cite{C}, as subspaces of structured matrices (e.g.\ Toeplitz, Hankel, banded), and in recent developments such as compressive sensing (e.g.\ Grassmannian dictionaries \cite{SH}, online matrix completion \cite{BNR}).

One of the most basic problems with subspaces is to define a notion of separation between them. The solution is well-known for subspaces of the same dimension $k$ in $\mathbb{R}^n$. These are points on the Grassmannian $\Gr(k,n)$, a Riemannian manifold, and the geodesic distance between them gives us an intrinsic distance. The \emph{Grassmann distance} is independent of the choice of coordinates and can be
readily related to principal angles and thus computed via the singular value decomposition (\textsc{svd}): For subspaces $\mathbf{A}, \mathbf{B} \in \Gr(k,n)$, form matrices $A, B \in \mathbb{R}^{n \times k}$ whose columns are their respective orthonormal bases, then
\begin{equation}\label{eq:grassdist}
d(\mathbf{A},\mathbf{B}) = \Bigl(\sum\nolimits_{i=1}^k \theta_i^2\Bigr)^{1/2},
\end{equation}
where $\theta_i = \cos^{-1} \bigl(\sigma_i(A^\mathsf{T} B)\bigr)$ is the $i$th principal angle between $\mathbf{A}$ and $\mathbf{B}$. This is the geodesic distance on the Grassmannian viewed as a  Riemannian manifold. There are many other common distances defined on Grassmannians ---  Asimov, Binet--Cauchy, chordal, Fubini--Study, Martin, Procrustes, projection, spectral (see Table~\ref{tab:distances}).

What if the subspaces are of different dimensions? In fact, if one examines the aforementioned applications, one invariably finds  that the most general settings for each of them  would fall under this situation. The restriction to equidimensional subspaces thus somewhat limits the utility of these applications. For example, the principal subspaces of two matrices $A$ and $B$ for a given noise level would typically be of different dimensions, since there is no reason to expect the number of singular values of $A$ above a given threshold to be the same as that of $B$.

As such one may also find many applications that involve distances between subspaces of different dimensions: numerical linear algebra \cite{BES, SS}, information retrieval \cite{ref2a, text}, facial recognition \cite{ref2b, face}, image classification \cite{ref2a, ref2b}, motion segmentation \cite{motion, ref2c, ref2d}, \textsc{eeg} signal analysis \cite{eeg}, mechanical engineering \cite{mech}, economics \cite{econ}, network analysis \cite{network}, blog spam detection \cite{blog}, and decoding colored barcodes \cite{bar}.

These applications are all based on two existing proposals for a distance between subspaces of different dimensions: The \emph{containment gap} \cite[pp.~197--199]{Kato} and the \emph{symmetric directional distance} \cite{SWF, WWF}. They are however somewhat ad hoc and bear little relation to the natural geometry of subspaces. Also, it is not clear what they are suppose to measure and neither restricts to the Grassmann distance when the subspaces are of the same dimension.  Our main objective is to show that there is an alternative definition that does generalize the Grassmann distance but our work will also shed light on these two distances.

%%%%
\subsection{Main Contributions} Our main result (see Theorem~\ref{thm1}) can be stated in simple linear algebraic terms: Given any two subspaces in $\mathbb{R}^n$, $\mathbf{A}$ of dimension $k$ and $\mathbf{B}$ of dimension $l$, assuming $k< l$ without loss of generality, the distance from $\mathbf{A}$ to the nearest $k$-dimensional subspace contained in $\mathbf{B}$ \emph{equals} the distance from $\mathbf{B}$ to the nearest $l$-dimensional subspace that contains $\mathbf{A}$. Their common value gives the distance between $\mathbf{A}$ and $\mathbf{B}$. Taking an algebraic geometric point-of-view:
\begin{enumerate}[\upshape ($\ast$)]
\item \label{quote}
\emph{The distance between subspaces of different dimensions is the distance between a point and a certain Schubert variety within the Grassmannian.}
\end{enumerate}
This distance has the following properties, established in Section~\ref{sec:main}:
\begin{enumerate}[\upshape (a)]
\item readily computable via \textsc{svd};
\item restricts to the usual Grassmann distance \eqref{eq:grassdist} for subspaces of the same dimension;
\item independent of the choice of local coordinates;
\item independent of the dimension of the ambient space (i.e., $n$);
\item\label{other} may be defined in conjunction with other common distances in Table~\ref{tab:distances}.
\end{enumerate}
We will see in Section~\ref{sec:existing} that the two existing notions of distance between subspaces of different dimensions are special cases of \eqref{other}.

Evidently, the word `distance' in ($\ast$) is used in the sense of a distance of a point to a set. For example, if a subspace is contained in another, then the distance between them is zero, even if they are distinct subspaces. Thus the distance in ($\ast$) is not a metric\footnote{We will see in Section~\ref{sec:dinfty} that this could be attributed to the fact that $\boldsymbol{\mathsf{Met}}$, the category of metric spaces and continuous contractions, does not admit coproduct.}. In Section~\ref{sec:metric}, we define a metric on the set of subspaces of all dimensions using an analogue of our main result: Given any two subspaces in $\mathbb{R}^n$, $\mathbf{A}$ of dimension $k$ and $\mathbf{B}$ of dimension $l$ with $k< l$, the distance from $\mathbf{A}$ to the furthest $k$-dimensional subspace contained in $\mathbf{B}$ \emph{equals} the distance from $\mathbf{B}$ to the furthest $l$-dimensional subspace that contains $\mathbf{A}$. Their common value gives a metric between $\mathbf{A}$ and $\mathbf{B}$. The most interesting metrics for subspaces of different dimensions can be found in Table~\ref{tab:metrics}.

In Section~\ref{sec:volume}, we obtain a volumetric analogue of our main result: Given two arbitrary subspaces in $\mathbb{R}^n$, $\mathbf{A}$ of dimension $k$ and $\mathbf{B}$ of dimension $l$ with $k < l$, we show that the probability a random $l$-dimensional subspace contains  $\mathbf{A}$ \emph{equals} the probability a random $k$-dimensional subspace is contained in $\mathbf{B}$.

The far-reaching work \cite{EAS} popularized the basic \emph{differential geometry} of Stiefel and Grassmannian manifolds by casting the discussions concretely in terms of matrices. Subsequent works, notably \cite{AMSbook, AMS, AMSV}, have further enriched this concrete matrix-based approach. A secondary objective of our article is to do the same for the basic \emph{algebraic geometry} of Grassmannians. In particular, we introduce some of the objects in Table~\ref{tab:objects} to an applied and computational mathematics readership. The proofs of our main results essentially use only the \textsc{svd}. Everything else is explained within the article and accessible to anyone willing to accept a small handful of unfamiliar terminologies and facts on faith.
\begin{table*}[ht]
\caption{Grassmannian and friends}
\label{tab:objects}
\vspace*{-1.5ex}
\renewcommand{\arraystretch}{1.75}\small
\begin{tabular}{p{0.29\textwidth}p{0.18\textwidth}p{0.43\textwidth}p{0.7\textwidth}}
\emph{Grassmannian} & $\Gr(k,n)$ & models $k$-dimensional subspaces in $\mathbb{R}^n$ & \S\ref{sec:Grass} \\
\emph{Infinite  Grassmannian} & $\Gr(k,\infty)$ & models $k$-dimensional subspaces regardless of ambient space & \S\ref{sec:infty} \\
\emph{Doubly-infinite Grassmannian} & $\Gr(\infty,\infty)$ & models subspaces of all dimensions regardless of ambient space & \S\ref{sec:dinfty} \\
\emph{Flag variety} & $\Flag(k_1,\dots,k_m,n)$ & models nested sequences of subspaces in $\mathbb{R}^n$; $\Flag(k,n) = \Gr(k,n)$ & \S\ref{sec:schubert} \\
\emph{Schubert variety} & $\Omega(\mathbf{X}_1,\dots,\mathbf{X}_m,n)$ & `linearly constrained' subset of $\Gr(k,n)$& \S\ref{sec:schubert}%, \S\ref{sec:matrix}
\end{tabular}
\end{table*}

\section{Grassmannian of linear subspaces}\label{sec:Grass}

We will selectively review some basic properties of the Grassmannian. The differential geometric perspectives are drawn from \cite{Husemoller, MS}, the more concrete matrix-theoretic view from \cite{AMS, EAS, Wong}, and the computational aspects from \cite{GVL}.

We fix the ambient space $\mathbb{R}^n$. A \emph{$k$-plane}  is a $k$-dimensional subspace of $\mathbb{R}^n$. A \emph{$k$-frame} is an ordered orthonormal basis of a $k$-plane, regarded as an $n \times k$ matrix whose columns $a_1,\dots, a_k$ are the orthonormal basis vectors. A \emph{flag} is a strictly increasing sequence of nested subspaces, $\mathbf{X}_0\subset \mathbf{X}_1\subset \cdots \subset \mathbf{X}_m\subset  \mathbb{R}^n$; it is \emph{complete} if $m=n$.

We write $\Gr(k,n)$ for the \emph{Grassmannian} of $k$-planes in $\mathbb{R}^n$, $\V(k,n)$ for the \emph{Stiefel manifold} of orthonormal $k$-frames, and $\O(n) \coloneqq \V(n,n)$ for the \emph{orthogonal group} of $n \times n$ orthogonal matrices. $\V(k,n)$ may be regarded as a homogeneous space,
\[
\V(k,n) \cong \O(n)/\O(n-k),
\]
or more concretely as the set of $n \times k$ matrices with orthonormal columns.

There is a \emph{right action} of $\O(k)$ on $\V(k,n)$: For $Q\in \O(k)$ and $A \in \V(k,n)$, the action  yields $AQ \in \V(k,n)$ and the resulting homogeneous space is $\Gr(k,n)$, i.e.,
\begin{equation}\label{eq:homo}
\Gr(k,n) \cong \V(k,n)/\O(k) \cong \O(n)/\bigl(\O(n-k) \times \O(k)\bigr).
\end{equation}
In this picture, a subspace $\mathbf{A} \in \Gr(k,n)$ is identified with an equivalence class comprising all its $k$-frames $\{ AQ \in \V(k,n): Q \in \O(k)\}$. Note that $\operatorname{span}(AQ) = \operatorname{span}(A)$ for all $Q \in \O(k)$.

There is a \emph{left action} of $\O(n)$ on $\Gr(k,n)$: For any $Q \in \O(n)$ and $\mathbf{A} = \operatorname{span}(A) \in \Gr(k,n)$ where $A$ is a $k$-frame of $\mathbf{A}$, the action yields
\begin{equation}\label{eq:action}
Q \cdot \mathbf{A}\coloneqq \operatorname{span}(QA) \in \Gr(k,n).
\end{equation}
This action is transitive as any $k$-plane can be rotated onto any other $k$-plane by some $Q \in \O(n)$. A $k$-plane  $\mathbf{A} \in \Gr(k,n)$ will be denoted in boldface; the corresponding italized letter $A =[a_1,\dots,a_k] \in \V(k,n)$ will denote a $k$-frame of $\mathbf{A}$.

$\Gr(k,n)$ and $\V(k,n)$ are smooth manifolds of dimensions  $k(n-k)$ and $nk-k(k+1)/2$ respectively.  As a set of $n\times k$ matrices, $\V(k,n)$ is a submanifold of $\mathbb{R}^{n\times k}$ and inherits a Riemannian metric  from the Euclidean metric on $\mathbb{R}^{n\times k}$, i.e., given $A=[a_1,\dots, a_k]$ and $B=[b_1,\dots, b_k]$ in $T_X \V(k,n)$, the tangent space at $X \in \V(k,n)$, the Riemannian metric $g$ is defined by $g_X(A,B)=\sum_{i=1}^k a_i^\mathsf{T} b_i= \tr (A^\mathsf{T}B)$. As $g$ is invariant under the action of $\O(k)$, it descends to a Riemannian metric on $\Gr(k,n)$ and in turn induces a geodesic distance  on $\Gr(k,n)$ which we define below.

Let $\mathbf{A} \in \Gr(k,n)$ and $\mathbf{B} \in \Gr(l,n)$ respectively. Let $r \coloneqq \min (k,l)$. The $i$th \emph{principal vectors} $(p_{i},q_{i})$, $i=1,\dots,r$, are defined recursively as solutions to the optimization problem%
\[%
\begin{tabular}
[c]{rcl}%
maximize &  & $p^{\mathsf{T}}q$\\
subject to 
&  & $p\in \mathbf{A},\; p^{\mathsf{T}}p_{1}=\dots=p^{\mathsf{T}}p_{i-1}=0,\;\lVert p\rVert=1,$\\
&  & $q\in \mathbf{B}, \; q^{\mathsf{T}}q_{1}=\dots=q^{\mathsf{T}}q_{i-1}=0,\;\lVert q\rVert=1,$%
\end{tabular}
\]
for $i=1,\dots,r$. The \emph{principal angles} are then defined by%
\[
\cos\theta_{i}=p_{i}^\mathsf{T}q_{i}, \quad i = 1,\dots,r.
\]
Clearly $0\le \theta_{1}\leq\dots\leq\theta_r \le \pi/2$. We will let $\theta_i(\mathbf{A},\mathbf{B})$ denote the $i$th principal angle between $\mathbf{A} \in \Gr(k,n)$ and $\mathbf{B} \in \Gr(l,n)$. 

Principal vectors and principal angles may be readily computed using \textsc{qr} and \textsc{svd} \cite{BG, GVL}. Let $A = [a_1,\dots,a_k]$ and $B =[ b_1,\dots, b_l]$ be orthonormal bases and let
\begin{equation}\label{eq:svd}
A^\mathsf{T} B = U\Sigma V^\mathsf{T}% = U_1 \Sigma_1 V_1^\mathsf{T}
\end{equation}
be the full \textsc{svd} of $A^\mathsf{T} B$, i.e., $U \in \O(k)$, $V \in \O(l)$, $\Sigma = \left[\begin{smallmatrix}\Sigma_1 & 0 \\ 0 & 0 \end{smallmatrix}\right]\in \mathbb{R}^{k \times l}$ with $\Sigma_1 = \Diag(\sigma_1,\dots,\sigma_r) \in \mathbb{R}^{r \times r}$ where $\sigma_1 \ge \dots \ge \sigma_{r} \ge 0$.

The principal angles $\theta_{1}\leq\dots\leq\theta_{r}$ are given by
\begin{equation}\label{eq:angles}
\theta_i = \cos^{-1} \sigma_i, \quad i = 1,\dots,r.
\end{equation}
It is customary to write $A^\mathsf{T} B = U (\cos\Theta) V^\mathsf{T}$, where  $\Theta = \Diag (\theta_1,\dots,\theta_{r},1,\dots,1) \in \mathbb{R}^{k \times l}$ and $\Theta_1 = \Diag (\theta_1,\dots,\theta_{r}) \in \mathbb{R}^{r \times r}$. Consider the column vectors,
\[
AU = [p_1,\dots,p_k], \quad BV = [q_1,\dots,q_l].
\]
The principal vectors are given by $(p_1,q_1),\dots,(p_r, q_r)$. Strictly speaking, principal vectors come in pairs but we will also call the vectors $p_{r+1},\dots,p_k$ (if $r = l < k$) or $q_{r+1},\dots,q_l$ (if $r = k < l$) principal vectors for lack of a better term.

We will use the following fact from \cite[Theorem~6.4.2]{GVL}.
\begin{proposition}\label{prop:angles}
Let $r = \min (k,l)$ and $\theta_1,\dots,\theta_{r}$ and $(p_1,q_1),\dots,(p_{r}, q_{r})$ be the principal angles and principal vectors between $\mathbf{A} \in \Gr(k,n)$ and $\mathbf{B} \in \Gr(l,n)$ respectively. If $m < r$ is such that
$1 = \cos \theta_1 = \dots = \cos \theta_m > \cos \theta_{m+1}$,
then
\[
\mathbf{A} \cap \mathbf{B} = \operatorname{span} \{p_1,\dots, p_m \} =  \operatorname{span} \{q_1,\dots, q_m \}.
\]
\end{proposition}
If $k = l$, the geodesic distance between $\mathbf{A}$ and $\mathbf{B}$ in $\Gr(k,n)$ is called the \emph{Grassmann distance} and is given by
\begin{equation}\label{eq:grassdist1}
d_{\Gr(k,n)}(\mathbf{A},\mathbf{B})=\Bigl(\sum\nolimits_{i=1}^k \theta_i^2\Bigr)^{1/2} = \lVert\cos^{-1}\Sigma\rVert_{F}.
\end{equation}
An explicit expression for the geodesic \cite{AMS}  connecting $\mathbf{A}$ to $\mathbf{B}$ on $\Gr(k,n)$ that minimizes the Grassmann distance is given by $\gamma : [0,1] \to \Gr(k,n)$,
\begin{equation}\label{eq:geodesic}
\gamma(t) = \operatorname{span}( AU\cos t\Theta+ Q\sin t \Theta )
\end{equation}
where $M= Q(\tan\Theta) U^{\mathsf{T}}$ is  a condensed \textsc{svd}  of the matrix
\[
M \coloneqq (I - AA^\mathsf{T})B(A^\mathsf{T} B)^{-1} \in \mathbb{R}^{n \times k}
\]
and where $U \in \O(k)$ and $\Theta=\Diag(\theta_1,\dots,\theta_k)\in \mathbb{R}^{k\times k}$ are as in \eqref{eq:svd} and \eqref{eq:angles}. Note that if $\cos\Theta=\Sigma$, then $\tan\Theta=(\Sigma^{-2}-I)^{1/2}$.  Also, $\gamma(0) =\mathbf{A}$ and $\gamma(1) = \mathbf{B}$.

Aside from the Grassmann distance, there are many well-known distances between subspaces \cite{BN,DD,DHST,EAS,HamLee2008}. We present some of these in Table~\ref{tab:distances}.
\begin{table*}[ht]
\caption{Distances on $\Gr(k,n)$ in terms of principal angles and orthonormal bases.}
\label{tab:distances}
\centering
\vspace*{-1.5ex}
\renewcommand{\arraystretch}{1.75}
\begin{tabular}{lll}
 & \emph{Principal angles} & \emph{Orthonormal bases}\\
Asimov & $d^{\alpha}_{\Gr(k,n)}(\mathbf{A}, \mathbf{B}) =  \theta_k$ & $\cos^{-1} \lVert A^\mathsf{T} B \rVert_2$\\
Binet--Cauchy & $d^{\beta}_{\Gr(k,n)}(\mathbf{A}, \mathbf{B}) = \left(1 - \prod\nolimits_{i=1}^k\cos^2\theta_i\right)^{1/2}$ & $(1 - (\det A^\mathsf{T} B)^2)^{1/2}$\\
Chordal & $d^{\kappa}_{\Gr(k,n)}(\mathbf{A}, \mathbf{B}) = \left(\sum\nolimits_{i=1}^k\sin^2\theta_i\right)^{1/2}$ & $\frac{1}{\sqrt{2}} \lVert AA^\mathsf{T} - BB^\mathsf{T} \rVert_F$\\
Fubini--Study & $d^{\phi}_{\Gr(k,n)}(\mathbf{A}, \mathbf{B}) = \cos^{-1} \left(\prod\nolimits_{i=1}^k\cos \theta_i\right)$ & $\cos^{-1} \lvert \det A^\mathsf{T} B \rvert$\\
Martin & $d^{\mu}_{\Gr(k,n)}(\mathbf{A}, \mathbf{B}) = \left( \log \prod\nolimits_{i=1}^k 1/\cos^2 \theta_i \right)^{1/2}$ & $(-2 \log \det A^\mathsf{T} B)^{1/2}$\\
Procrustes & $d^{\rho}_{\Gr(k,n)}(\mathbf{A}, \mathbf{B}) = 2\left(\sum\nolimits_{i=1}^k\sin^2(\theta_i/2)\right)^{1/2}$ & $\lVert AU - BV \rVert_F$\\
Projection & $d^{\pi}_{\Gr(k,n)}(\mathbf{A}, \mathbf{B}) = \sin \theta_k$ & $\lVert AA^\mathsf{T} - BB^\mathsf{T} \rVert_2$\\
Spectral & $d^{\sigma}_{\Gr(k,n)}(\mathbf{A}, \mathbf{B}) = 2\sin (\theta_k/2)$ & $\lVert AU - BV \rVert_2$
\end{tabular}
\end{table*}

The value $\sin \theta_1$ is sometimes called the \emph{max correlation distance} \cite{HamLee2008} or spectral distance \cite{DHST} but it is not a distance in the sense of a metric (can be zero for a pair of distinct subspaces) and thus not listed. The spectral distance $d^{\sigma}_{\Gr(k,n)}$ is also called chordal $2$-norm distance \cite{BN}. 
For each distance in Table~\ref{tab:distances} defined for equidimensional $\mathbf{A}$ and $\mathbf{B}$, Theorem~\ref{thm:othermetrics}  provides a corresponding version for when $\dim \mathbf{A} \ne \dim \mathbf{B}$.

The fact that all these distances in Table~\ref{tab:distances} depend on the principal angles is not a coincidence --- the result \cite[Theorem~3]{Wong} implies the following.
\begin{theorem}\label{thm:wong}
Any notion of distance between $k$-dimensional subspaces in $\mathbb{R}^n$ that depends only on the relative positions of the subspaces, i.e., invariant under any rotation in $\O(n)$, must be a function of their principal angles. To be more specific, if a distance $d : \Gr(k,n) \times \Gr(k,n) \to [0, \infty)$ satisfies
\[
d(Q \cdot \mathbf{A}, Q \cdot \mathbf{B}) =d(\mathbf{A}, \mathbf{B}),
\]
for all $\mathbf{A}, \mathbf{B} \in \Gr(k,n)$ and all $Q \in \O(n)$, where the action is as defined in \eqref{eq:action}, then $d$ must be a function of $\theta_i(\mathbf{A},\mathbf{B})$, $i=1,\dots,k$.
\end{theorem}
We will next introduce the \emph{infinite Grassmannian} $\Gr(k,\infty)$ to show that these distances between subspaces are independent of the dimension of their ambient space.

%%%%
\section{The Infinite Grassmannian}\label{sec:infty}

One way of defining a distance between $\mathbf{A} \in \Gr(k,n)$ and $\mathbf{B}\in \Gr(l,n)$ where $k \ne l$ is to first isometrically embed $\Gr(k,n)$ and $\Gr(l,n)$ into an ambient Riemannian manifold and then define the distance between $\mathbf{A}$ and $\mathbf{B}$ as their distance in the ambient space. This approach is taken in \cite{CHS, SLLM}, via an isometric embedding of $\Gr(0,n), \Gr(1,n), \dots, \Gr(n,n )$ into a sphere of dimension $(n-1)(n+2)/2$. Such a distance suffers from two shortcomings: It is not intrinsic to the Grassmannian and it depends on both the embedding and the ambient space.

The distance that we propose in Section~\ref{sec:main} will depend only on the intrinsic distance of the Grassmannian and is independent of $n$, i.e., a $k$-plane $\mathbf{A}$ and an $l$-plane $\mathbf{B}$ in  $\mathbb{R}^n$ will have the same distance if we regard them as subspaces in $\mathbb{R}^m$ for any $m \ge \min(k,l)$. We will first establish this for the special case $k = l$.

Consider the inclusion map $\iota_n:\mathbb{R}^n\to \mathbb{R}^{n+1}$, $\iota_n(x_1,\dots, x_n)=(x_1,\dots, x_n,0)$. It is easy to see that $\iota_n$ induces a natural inclusion of $\Gr(k,n)$ into $\Gr(k,n+1)$ which we will also denote by $\iota_n$. For any $m > n$, composition of successive natural inclusions gives the inclusion $\iota_{nm} : \Gr(k,n) \to \Gr(k,m)$, where
$\iota_{nm} \coloneqq \iota_n \circ \iota_{n+1} \circ \dots \circ \iota_{m-1}$.
To be more concrete, if $A \in \mathbb{R}^{n \times k}$ has orthonormal columns, then
\begin{equation}\label{eq:iota}
\iota_{nm} : \Gr(k,n) \to \Gr(k,m), \qquad
\operatorname{span} (A) \mapsto \operatorname{span} \left( \begin{bmatrix}A \\ 0 \end{bmatrix}  \right), 
\end{equation}
where the zero block matrix is $(m-n) \times k$ so that $\left[ \begin{smallmatrix}A \\ 0 \end{smallmatrix} \right] \in \mathbb{R}^{m \times k}$.

For a fixed $k$, the family of Grassmannians $\{ \Gr(k,n): n\in \mathbb{N}, \; n \ge k \}$ together with the inclusion maps $\iota_{nm}:\Gr(k,n)\to \Gr(k,m)$ for $m > n$ form a direct system. The \emph{infinite Grassmannian} of $k$-planes  is defined to be the direct limit of this system in the category of topological spaces and denoted by
\[
\Gr(k,\infty)\coloneqq \varinjlim\Gr(k,n).
\]
Those unfamiliar with the notion of direct limits may simply take
\[
\Gr(k, \infty) = \bigcup\nolimits_{n=k}^\infty \Gr(k,n),
\]
where we regard $\Gr(k,n) \subset \Gr(k,n+1)$ by identifying $\Gr(k,n)$ with $\iota_n\bigl(\Gr(k,n)\bigr)$. With this identification, we no longer need to distinguish between $\mathbf{A} \in \Gr(k,n)$ and its image $\iota_{n}(\mathbf{A})\in \Gr(k,n+1)$ and may regard $\mathbf{A} \in \Gr(k,m)$ for all $m > n$.

We now define a distance $d_{\Gr(k, \infty)}$ on $\Gr(k,\infty)$ that is consistent with the Grassmann distance on $\Gr(k,n)$ for all $n$ sufficiently large.
\begin{lemma}\label{lem:infty}
The natural inclusion $\iota_n:\Gr(k,n)\to \Gr(k,n+1)$ is isometric, i.e.,
\begin{equation}\label{eq:iso}
d_{\Gr(k,n)}(\mathbf{A},\mathbf{B})=d_{\Gr(k,n+1)}\bigl(\iota_n(\mathbf{A}),\iota_n(\mathbf{B})\bigr).
\end{equation}
Repeated applications of \eqref{eq:iso} yields
\begin{equation}\label{eq:iso1}
d_{\Gr(k,n)}(\mathbf{A},\mathbf{B})=d_{\Gr(k,m)}\bigl(\iota_{nm}(\mathbf{A}),\iota_{nm}(\mathbf{B})\bigr)
\end{equation}
for all $m > n$ and if we identify $\Gr(k,n)$ with $\iota_n\bigl(\Gr(k,n)\bigr)$, we may rewrite \eqref{eq:iso1} as
\[
d_{\Gr(k,n)}(\mathbf{A},\mathbf{B})=d_{\Gr(k,m)}(\mathbf{A},\mathbf{B})
\]
for all $m > n$.
\end{lemma}
\begin{proof} If $a\in \mathbb{R}^n$, we write $\hat{a} = \left[\begin{smallmatrix}
a \\
0
\end{smallmatrix}  \right]\in \mathbb{R}^{n+1}$. Let $A = [a_1,\dots, a_k]$ and $B = [b_1,\dots,b_k]$ be any orthonormal bases of $\mathbf{A}$ and $\mathbf{B}$ respectively. By the definition of  $\iota_n$, $\iota_n(\mathbf{A})$ is the subspace in $\mathbb{R}^{n+1}$ spanned by an orthonormal basis that we will denote by $\iota_n(A) \coloneqq [\hat{a}_1,\dots, \hat{a}_k] \in \mathbb{R}^{(n+1) \times k}$. Hence we have
\[
\iota_n(A)^\mathsf{T} \iota_n(B)
=
\begin{bmatrix}
A^\mathsf{T} B \\
0
\end{bmatrix}.
\]
By the expression for Grassmann distance in \eqref{eq:grassdist1}, we see that \eqref{eq:iso} must hold.
\end{proof}

Since the inclusion of  $\Gr(k,n)$ in  $\Gr(k,n+1)$ is isometric, a geodesic in $\Gr(k,n)$ remains a geodesic in $\Gr(k,n+1)$. Given $\mathbf{A},\mathbf{B} \in \Gr(k, \infty)$, there must exist some $n$ sufficiently large so that both $\mathbf{A},\mathbf{B} \in \Gr(k, n)$ and in which case we define the distance between $\mathbf{A}$ and $\mathbf{B}$ in $ \Gr(k, \infty)$ to be
\[
d_{\Gr(k, \infty)}(\mathbf{A},\mathbf{B}) \coloneqq d_{\Gr(k, n)}(\mathbf{A},\mathbf{B}).
\]
By Lemma~\ref{lem:infty}, this value is independent of our choice of $n$ and is the same for all $m \ge n$. In particular, $d_{\Gr(k, \infty)}$ is well-defined and yields a distance on $\Gr(k,\infty)$. We summarize these observations below.
\begin{corollary}\label{cor:infty}
The Grassmann distance between two $k$-planes in $\Gr(k,n)$ is the geodesic distance in $\Gr(k,\infty)$ and is therefore independent of $n$. Also, the expression \eqref{eq:geodesic}  for a distance minimizing geodesic in $\Gr(k,n)$ extends to $\Gr(k,\infty)$. 
\end{corollary}

Lemma~\ref{lem:infty} also holds for other distances on $\Gr(k,n)$  in Table~\ref{tab:distances}, allowing us to define them on $\Gr(k, \infty)$.
\begin{lemma}\label{lem:inclusion}
For all $m > n$, the inclusion $\iota_{nm}:\Gr(k,n)\to \Gr(k,m)$ is isometric when $\Gr(k,n)$ and $\Gr(k,m)$ are both equipped with the same distance in Table~\ref{tab:distances}, i.e., 
\[
d^*_{\Gr(k,n)}(\mathbf{A},\mathbf{B})=d^*_{\Gr(k,m)}(\iota_{nm}(\mathbf{A}),\iota_{nm}\bigl(\mathbf{B})\bigr),
\]
$* = \alpha, \beta, \kappa, \phi, \mu, \rho, \pi, \sigma$. Consequently $d^{\ast}_{\Gr(k,\infty)}$ is well-defined.
\end{lemma}
\begin{proof}
$d^{\ast}_{\Gr(k,n)}(\mathbf{A},\mathbf{B})$ and $d^{\ast}_{\Gr(k,n+1)}\bigl(\iota_n(\mathbf{A}),\iota_n\bigl(\mathbf{B})\bigr)$ depend only on the principal angles between $\mathbf{A}$ and $\mathbf{B}$, so the distance remains unchanged under $\iota_n$. Repeated applications to $\iota_n \circ \iota_{n+1} \circ \dots \circ \iota_{m-1} = \iota_{nm}$ yield the required isometry.
\end{proof}

%%%%
\section{Distances between subspaces of different dimensions}\label{sec:main}

We now address our main problem. The proposed notion of distance will be that of a point $x \in X$ to a set $S \subset X$ in a metric space $(X,d)$. Recall that this is defined by $d(x, S) \coloneqq \inf\{d(x,y) : y  \in S \}$. For us, $X$ is a Grassmannian, therefore compact, and so $d(x,S)$ is finite. Also, $S$ will be a closed subset and so we write $\min$ instead of $\inf$. We will introduce two possible candidates for $S$.

\begin{definition}\label{def:Omega}
Let $k,l, n \in \mathbb{N}$ be  such that $k\le l \le n$. For any $\mathbf{A}\in \Gr(k,n)$ and $\mathbf{B}\in \Gr(l,n)$, we define the subsets
\[
\Omega_{+}(\mathbf{A})\coloneqq\bigl\{\mathbf{X}\in \Gr(l,n) : \mathbf{A}\subseteq \mathbf{X}\bigr\},\quad
\Omega_{-}(\mathbf{B})\coloneqq\bigl\{\mathbf{Y}\in \Gr(k,n) : \mathbf{Y} \subseteq \mathbf{B}\bigr\}.
\]
We will call $\Omega_{+}(\mathbf{A})$ the \emph{Schubert variety of $l$-planes containing $\mathbf{A}$} and $\Omega_{-}(\mathbf{B})$ the \emph{Schubert variety of $k$-planes contained in $\mathbf{B}$}.
\end{definition}
As we will see in Section~\ref{sec:schubert}, $\Omega_{+}(\mathbf{A})$ and $\Omega_{-}(\mathbf{B})$ are indeed Schubert varieties and therefore closed subsets of $\Gr(l,n)$ and $\Gr(k,n)$ respectively. Furthermore, $\Omega_{+}(\mathbf{A})$ and $\Omega_{-}(\mathbf{B})$ are uniquely determined by $\mathbf{A}$ and $\mathbf{B}$  (see Proposition~\ref{prop:Omega1}) and may be regarded as `sub-Grassmannians'  of $\Gr(l,n)$ and $\Gr(k,n)$ respectively (see Proposition~\ref{prop:OmegaGrass}).

How could one define the distance between a subspace $\mathbf{A}$ of dimension $k$ and a subspace $\mathbf{B}$ of dimension $l$ in $\mathbb{R}^n$ when $k \ne l$? We may assume $k < l\le n$ without loss of generality. In which case a very natural solution is to define the required distance $\delta(\mathbf{A}, \mathbf{B}) $ as that between the $k$-plane $\mathbf{A}$ and the closest $k$-plane $\mathbf{Y}$ contained in $\mathbf{B}$, measured within $\Gr(k,n)$. In other words, we want the Grassmann distance from $\mathbf{A}$ to the closed subset $\Omega_{-}(\mathbf{B})$,
\begin{equation}\label{eq:minus}
\delta(\mathbf{A}, \mathbf{B}) \coloneqq d_{\Gr(k,n)}\bigl(\mathbf{A}, \Omega_{-}\bigl(\mathbf{B})\bigr) = \min\bigl\{ d_{\Gr(k,n)}(\mathbf{A},\mathbf{Y}) :  \mathbf{Y}\in \Omega_{-}(\mathbf{B})\bigr\}. 
\end{equation}
This has the advantage of being intrinsic --- the distance $\delta(\mathbf{A}, \mathbf{B}) $  is measured in $d_{\Gr(k,n)}$ and is defined wholly  within $\Gr(k,n)$ without any embedding of $\Gr(k,n)$ into an arbitrary ambient space. Furthermore, by the property of $d_{\Gr(k,n)}$ in Corollary~\ref{cor:infty},  $\delta(\mathbf{A}, \mathbf{B}) $ does not depend on $n$ and takes the same value for any $m \ge n$. We illustrate this in Figure~\ref{fig:sphere}: The sphere is intended to be a depiction of $\Gr(1,3)$ though to be accurate antipodal points on the sphere should be identified.
\begin{figure}[ht]
\centering
\begin{tikzpicture}%[scale=0.86]
% point A
\draw[red][fill] (-4,0) circle[radius=0.05];
\node[red] [left] at (-4.08,0) {$\mathbf{A}$};
% sphere 
    \draw[blue] (0,4) arc (90:270:2cm and 4cm);
    \draw[dashed] (0,4) arc (90:-90:2cm and 4cm);
    \draw (0,0) circle (4cm);
    \shade[ball color=blue!10!white,opacity=0.20] (0,0) circle (4cm);
    \node[black][below right] at (3,-3) {$\operatorname{Gr}(1,3)$};
% \Omega_{-}(\mathbf{B})
        \node[blue] [left] at (0,2.5) {$\Omega_{-}(\mathbf{B})$};
% geodesic
\node[above,magenta] at (-3,-1.25) {$\gamma$};
            \draw[magenta] (-4,0) arc (180:243:4cm and 2cm);
% nearest point
\draw[blue][fill] (-1.8,-1.8) circle[radius=0.05];
\node[blue] [below] at (-2,-1.9) {$\mathbf{X}$};
\end{tikzpicture}
\caption{Distance between a line $\mathbf{A}$ and a plane $\mathbf{B}$ in $\mathbb{R}^3$. $\mathbf{X}$ is closest to $\mathbf{A}$ among all lines in $\mathbf{B}$. The length of the geodesic $\gamma$ from $\mathbf{A}$ to $\mathbf{X}$ gives the distance.}
\label{fig:sphere}
\end{figure}
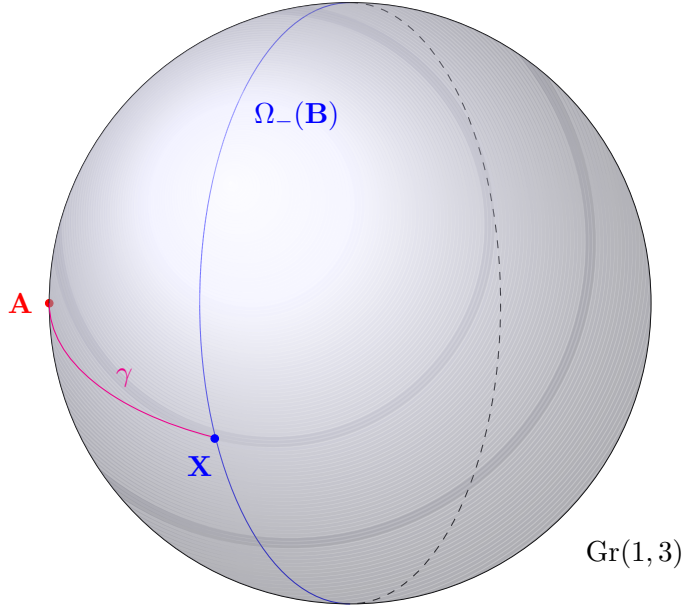

However, it is equally natural to define  $\delta(\mathbf{A}, \mathbf{B}) $  as the distance between the $l$-plane $\mathbf{B}$ and the closest $l$-plane $\mathbf{Y}$ containing $\mathbf{A}$, measured within $\Gr(l,n)$. In other words, we could have instead defined  it as the Grassmann distance from $\mathbf{B}$ to the closed subset $\Omega_{+}(\mathbf{A})$,
\begin{equation}\label{eq:plus}
\delta(\mathbf{A}, \mathbf{B}) \coloneqq d_{\Gr(l,n)}\bigl(\mathbf{B}, \Omega_{+}\bigl(\mathbf{A})\bigr) 
= \min\bigl\{ d_{\Gr(l,n)}(\mathbf{B},\mathbf{X}) :  \mathbf{X}\in \Omega_{+}(\mathbf{A})\bigr\}.
\end{equation}
It will have the same desirable features as the one in \eqref{eq:minus} except that the distance is now measured in $d_{\Gr(l,n)}$ and within $\Gr(l,n)$.

It turns out that the two values in \eqref{eq:minus} and \eqref{eq:plus} are equal, allowing us to define $\delta(\mathbf{A}, \mathbf{B})$ as their common value. We will establish this equality and the properties of $\delta(\mathbf{A}, \mathbf{B})$ in the remainder of this section. The results are summarized in  Theorem~\ref{thm1}. Our proof is constructive: In addition to showing the equality of \eqref{eq:minus} and \eqref{eq:plus}, it shows how one may explicitly find the closest points on Schubert varieties $\mathbf{X} \in \Omega_{-}(\mathbf{B})$ and $\mathbf{Y} \in \Omega_{+}(\mathbf{A})$ to  any given point in the respective Grassmannians.% (see Algorithm~\ref{alg:nearest}).
\begin{theorem}\label{thm1}
Let $\mathbf{A}$ be a subspace of dimension $k$ and $\mathbf{B}$ be a subspace of dimension $l$ in $\mathbb{R}^n$. Suppose $k \le l \le n$. Then
\begin{equation}\label{eq:equal}
 d_{\Gr(k,n)}\bigl(\mathbf{A}, \Omega_{-}\bigl(\mathbf{B})\bigr) =  d_{\Gr(l,n)}\bigl(\mathbf{B}, \Omega_{+}\bigl(\mathbf{A})\bigr).
\end{equation}
Their common value defines a distance $\delta(\mathbf{A}, \mathbf{B}) $ between the two subspaces with the following properties:
\begin{enumerate}[\upshape (i)]
\item\label{indp} $\delta(\mathbf{A}, \mathbf{B}) $  is independent of the dimension of the ambient space $n$ and is the same for all $n \ge l+1$;
\item\label{reduce} $\delta(\mathbf{A}, \mathbf{B}) $  reduces to the Grassmann distance between $\mathbf{A}$ and $\mathbf{B}$ when $k = l$;
\item\label{explicit} $\delta(\mathbf{A}, \mathbf{B}) $  may be computed explicitly as
\begin{equation}\label{eq:grassdist2}
\delta(\mathbf{A},\mathbf{B})=\Bigl(\sum\nolimits_{i=1}^{\min \{k,l\}}\theta_i(\mathbf{A},\mathbf{B})^2 \Bigr)^{1/2} 
\end{equation}
where $\theta_i(\mathbf{A},\mathbf{B})$ is the $i$th principal angle between $\mathbf{A}$ and $\mathbf{B}$, $i =1,\dots,\min(k,l)$.
\end{enumerate}
\end{theorem}

Rewriting \eqref{eq:equal} as
\[
\min_{\mathbf{X}\in \Omega_{+}(\mathbf{A})}d_{\Gr(l,n)}(\mathbf{X},\mathbf{B}) = \min_{\mathbf{Y} \in \Omega_{-}(\mathbf{B})} d_{\Gr(k,n)}(\mathbf{Y},\mathbf{A}),
\]
the equation says that the distance from $\mathbf{B}$ to the nearest $l$-dimensional subspace  that contains $\mathbf{A}$ equals the distance from $\mathbf{A}$ to the nearest $k$-dimensional subspace  contained in $\mathbf{B}$. This relation has several parallels. We will see that:
\begin{enumerate}[\upshape (a)]
\item the Grassmann distance may be replaced by any of the distances in Table~\ref{tab:distances} (see Theorem~\ref{thm:othermetrics});
\item `nearest' may be replaced by `furthest' and `min' above replaced by `max' when $n$ is sufficiently large (see Proposition~\ref{prop:prob4});
%\item `linear' may be replaced by  `affine' (see Section~\ref{sec:distaff});
\item `distance' may be replaced by `volume' with respect to the intrinsic uniform probability density on the Grassmannian (see Section~\ref{sec:volume}).
\end{enumerate}

We will prove Theorem~\ref{thm1} by way of the next two lemmas. 
\begin{lemma}\label{lem:ineq1}
Let $k \le l \le n$ be positive integers. Let $\delta :\Gr(k,n)\times \Gr(l,n)\to [0,\infty)$ be the function defined  by 
\[
\delta(\mathbf{A},\mathbf{B})=\Bigl(\sum\nolimits_{i=1}^k\theta_i^2\Bigr)^{1/2}
\]
where $\theta_i \coloneqq \theta_i(\mathbf{A},\mathbf{B})$, $i=1,\dots,k$. Then
\[
\delta(\mathbf{A},\mathbf{B})\ge d_{\Gr(l,n)}\bigl(\mathbf{B},\Omega_{+}(\mathbf{A})\bigr).
\]
\end{lemma}
\begin{proof}
%By Proposition~\ref{prop:angles}, we may assume that $\mathbf{A} \cap \mathbf{B} = \{0 \}$. 
It suffices to find an $\mathbf{X}\in \Omega_{+}(\mathbf{A})$ such that $\delta(\mathbf{A},\mathbf{B})=d_{\Gr(l,n)}(\mathbf{X},\mathbf{B})$. Let $(p_1,q_1),\dots, (p_k,q_k)$  be the principal vectors between $\mathbf{A}$ and $\mathbf{B}$. We will extend $q_1,\dots,q_k$ into an orthonormal basis of $\mathbf{B}$ by appending appropriate orthonormal vectors $q_{k+1},\dots,q_l$. The principal angles are given by $\theta_i=\cos^{-1} p_i^\mathsf{T} q_i$, $\lVert p_i \rVert =\lVert q_i \rVert=1$. %, and $p_i^\mathsf{T} p_j=0=q_i^\mathsf{T} q_j$ for $j=1,2,\dots,i-1$.
If we take $\mathbf{X} \in \Gr(l,n)$ to be the subspace spanned by $p_1,\dots,p_k,q_{k+1},\dots, q_l$, then %$\mathbf{Y}=\mathbf{X}\cap\mathbf{B}$ is spanned by $q_{k+1},\dots,q_l$.
\begin{align}\label{eq:ineq1}
d_{\Gr(l,n)}(\mathbf{X},\mathbf{B})
&= [(\cos^{-1} p_1^\mathsf{T} q_1)^2 + \dots +  (\cos^{-1} p_k^\mathsf{T} q_k)^2 \nonumber \\
&\qquad  +  (\cos^{-1}  q_{k+1}^\mathsf{T} q_{k+1})^2 + \dots + (\cos^{-1}  q_l^\mathsf{T} q_l)^2 ]^{1/2}\\
&= [ \theta_1^2 + \dots +  \theta_k^2 + 0^2 + \dots + 0^2]^{1/2} = \delta(\mathbf{A},\mathbf{B}). \nonumber 
\end{align}
%\[
%d_{\Gr(l,n)}(\mathbf{X},\mathbf{B})=d_{\Gr(k,n)}(\mathbf{X}/\mathbf{Y},\mathbf{B}/\mathbf{Y})=d_{\Gr(k,n)}(\mathbf{A},\mathbf{B}/\mathbf{Y})=\delta(\mathbf{A},\mathbf{B}).
%\]
%The last equality is obtained by definition of principal angles.
\end{proof}

We state the following well-known fact \cite[Corollary~3.1.3]{Horn} for easy reference and deduce a corollary that will be useful for Lemma~\ref{lem:ineq2}.
\begin{proposition}\label{prop:compare}
Let $k\le l\le n$ be positive integers. Suppose $B \in \mathbb{R}^{n\times l}$ and $B_k \in \mathbb{R}^{n\times k}$ is a submatrix obtained by removing any $l-k$ columns from $B$. Then the $i$th singular values satisfy $\sigma_i(B_k)\le \sigma_i(B)$ for $i=1,\dots, k$.
\end{proposition}
\begin{corollary}\label{cor:compare}
Let $B$ and $B_k$ be as in Proposition~\ref{prop:compare} and $\mathbf{B}$ and $\mathbf{B}_k$ be subspaces of $\mathbb{R}^n$ spanned by the column vectors of $B$ and $B_k$ respectively. Then for any subspace $\mathbf{A}$ of $\mathbb{R}^n$, the principal angles between the respective subspaces satisfy
\[
\theta_i(\mathbf{A},\mathbf{B})\le \theta_i(\mathbf{A},\mathbf{B}_k)
\]
for $i = 1,\dots,\min(\dim \mathbf{A}, \dim \mathbf{B}_k)$.
\end{corollary}
\begin{proof}
By appropriate orthogonalization if necessary, we may assume that $B$ and its submatrix $B_k$ are orthonormal bases of $\mathbf{B}$ and $\mathbf{B}_k$. Let $A$ be an orthonormal basis of $\mathbf{A}$. Then $\sigma_i(A^\mathsf{T}B)$ and $\sigma_i(A^\mathsf{T}B_k)$ take values in $[0,1]$. Since $\theta_i(\mathbf{A},\mathbf{B})= \cos^{-1}(\sigma_i\bigl(A^\mathsf{T}B)\bigr)$ and $\cos^{-1}$  is monotone decreasing in $[0,1]$, the result follows from
$\sigma_i(A^\mathsf{T}B)\ge \sigma_i(A^\mathsf{T}B_k)$, by Proposition~\ref{prop:compare} applied to the submatrix $A^\mathsf{T}B_k$ of $A^\mathsf{T}B$.
\end{proof}
\begin{lemma}\label{lem:ineq2}
Let $\mathbf{A}$, $\mathbf{B}$ be as in Lemma~\ref{lem:ineq1}. Then
$\delta(\mathbf{A}, \mathbf{B})\le d_{\Gr(k,n)}\bigl(\mathbf{A},\Omega_{-}(\mathbf{B})\bigr)$.
\end{lemma}
\begin{proof}
Let $\mathbf{Y} \in \Omega_{-}(\mathbf{B})$. Then $\mathbf{Y}$ is a $k$-dimensional subspace contained in $\mathbf{B}$ and in the notation of Corollary~\ref{cor:compare}, we may write $\mathbf{Y} = \mathbf{B}_{k}$. By the same corollary we get
$\theta_i(\mathbf{A},\mathbf{B})\le \theta_i(\mathbf{A},\mathbf{Y})$
for $i = 1, \dots, k$. Hence
\begin{equation}\label{eq:ineq2}
\delta(\mathbf{A},\mathbf{B})=\Bigl(\sum\nolimits_{i=1}^{k}\theta_i(\mathbf{A},\mathbf{B})^2\Bigr)^{1/2} 
\le \Bigl(\sum\nolimits_{i=1}^{k}\theta_i(\mathbf{A},\mathbf{Y})^2\Bigr)^{1/2}=d_{\Gr(k,n)}(\mathbf{A},\mathbf{Y}). 
\end{equation}
The desired inequality follows since this holds for arbitrary $\mathbf{Y} \in \Omega_{-}(\mathbf{B})$.
\end{proof}

\begin{proof}[Proof of Theorem~\ref{thm1}]
Recall that Grassmannians satisfy an isomorphism 
\[
\Gr(k,n)\cong \Gr(n-k,n)
\]
that takes a $k$-plane $\mathbf{Y}$ to the $(n-k)$-plane $\mathbf{Y}^{\perp}$ of linear forms vanishing on $\mathbf{Y}$. It is easy to see that this isomorphism is an isometry. Using this isometric isomorphism, together with Lemma~\ref{lem:ineq1} and Lemma~\ref{lem:ineq2}, we can immediately deduce that
\[
\delta(\mathbf{A}, \mathbf{B}) \le d_{\Gr(k,n)}\bigl(\mathbf{A},\Omega_{-}\bigl(\mathbf{B})\bigr)=d_{\Gr(n-k,n)}\bigl(\mathbf{A^{\perp}},\Omega_{+}\bigl(\mathbf{B}^{\perp})\bigr)\le \delta(\mathbf{A}^{\perp},\mathbf{B}^{\perp}).
\]
On the other hand, by results in \cite{Knyazev}, we have  $\delta(\mathbf{A}, \mathbf{B}) = \delta(\mathbf{A}^{\perp},\mathbf{B}^{\perp})$
and hence
\[
\delta(\mathbf{A}, \mathbf{B}) = d_{\Gr(k,n)}\bigl(\mathbf{A},\Omega_{-}(\mathbf{B})\bigr).
\]
Similarly we can obtain
\[
\delta(\mathbf{A}, \mathbf{B}) = d_{\Gr(l,n)}\bigl(\mathbf{B},\Omega_{+}(\mathbf{A})\bigr).
\]
Hence we have the required equalities \eqref{eq:equal} and \eqref{eq:grassdist2} in Theorem~\ref{thm1}. Property~\eqref{reduce} is obvious from \eqref{eq:grassdist2} and Property~\eqref{indp} follows from Lemma~\ref{lem:infty}.
\end{proof}
The proof of Lemma~\ref{lem:ineq1} provides a simple way to find a point $\mathbf{X} \in \Omega_{+}(\mathbf{A})$ that realizes the distance $d_{\Gr(l,n)}\bigl(\mathbf{B},\Omega_{+}(\mathbf{A})\bigr)=\delta(\mathbf{A},\mathbf{B})$. Similarly we may explicitly determine a point $\mathbf{Y} \in \Omega_{-}(\mathbf{B})$ that realizes the distance $d_{\Gr(k,n)}\bigl(\mathbf{A},\Omega_{-}(\mathbf{B})\bigr)=\delta(\mathbf{A},\mathbf{B})$. %We state these in Algorithm~\ref{alg:nearest}.

One might wonder whether or not Theorem~\ref{thm1} still holds if we replace $d_{\Gr(k,n)}$ by other distance functions described in Table~\ref{tab:distances}. The answer is yes.
\begin{theorem}\label{thm:othermetrics}
Let $k \le l \le n$. Let $\mathbf{A}\in \Gr(k,n)$ and $\mathbf{B}\in\Gr(l,n)$. Then
\[
 d^*_{\Gr(k,n)}\bigl(\mathbf{A}, \Omega_{-}(\mathbf{B})\bigr) =  d^*_{\Gr(l,n)}\bigl(\mathbf{B}, \Omega_{+}(\mathbf{A})\bigr),
\]
for $* = \alpha, \beta, \kappa, \phi, \mu, \rho, \pi, \sigma$. Their common value $\delta^*(\mathbf{A},\mathbf{B})$ is given by:
%\begin{align*}
%\delta^{\alpha}(\mathbf{A}, \mathbf{B}) &=  \theta_k, &\delta^{\beta}(\mathbf{A}, \mathbf{B}) &= \Bigl(1 - \prod\nolimits_{i=1}^k\cos^2\theta_i\Bigr)^{1/2},
%&\delta^{\kappa}(\mathbf{A}, \mathbf{B}) &= \Bigl(\sum\nolimits_{i=1}^k\sin^2\theta_i\Bigr)^{1/2},\\
%\delta^{\pi}(\mathbf{A}, \mathbf{B}) &= \sin \theta_k, &\delta^{\mu}(\mathbf{A}, \mathbf{B}) &= \Bigl( \log \prod\nolimits_{i=1}^k\frac{1}{\cos^2 \theta_i}\Bigr)^{1/2}, & \delta^{\phi}(\mathbf{A}, \mathbf{B}) &= \cos^{-1}\bigl(\prod\nolimits_{i=1}^k\cos \theta_i\Bigr),\\
%\delta^{\sigma}(\mathbf{A}, \mathbf{B}) &= 2\sin (\theta_k/2), & \delta^{\rho}(\mathbf{A}, \mathbf{B}) &= \Bigl(2\sum\nolimits_{i=1}^k\sin^2(\theta_i/2)\Bigr)^{1/2}, &&
%\end{align*}
\begin{align*}
\delta^{\alpha}(\mathbf{A}, \mathbf{B}) &=  \theta_k, &
\delta^{\beta}(\mathbf{A},\mathbf{B}) &= \Bigl(1 - \prod\nolimits_{i=1}^k\cos^2\theta_i\Bigr)^{1/2},\\
\delta^{\kappa}(\mathbf{A}, \mathbf{B}) &= \Bigl(\sum\nolimits_{i=1}^k\sin^2\theta_i\Bigr)^{1/2}, &
\delta^{\phi}(\mathbf{A}, \mathbf{B}) &= \cos^{-1}\bigl(\prod\nolimits_{i=1}^k\cos \theta_i\Bigr),\\
\delta^{\mu}(\mathbf{A}, \mathbf{B}) &= \Bigl( \log \prod\nolimits_{i=1}^k\frac{1}{\cos^2 \theta_i}\Bigr)^{1/2}, &
\delta^{\rho}(\mathbf{A}, \mathbf{B}) &= \Bigl(2\sum\nolimits_{i=1}^k\sin^2(\theta_i/2)\Bigr)^{1/2},\\
\delta^{\pi}(\mathbf{A}, \mathbf{B}) &= \sin \theta_k, &
\delta^{\sigma}(\mathbf{A}, \mathbf{B}) &= 2\sin (\theta_k/2),
\end{align*}
or more generally with $\min(k,l)$ in place of the index $k$ when we do not require $k \le l$.
\end{theorem}
\begin{proof}
This follows from observing that our proof of Theorem~\ref{thm1} only involves principal angles between $\mathbf{A}$ and $\mathbf{B}$ and the diffeomorphism between $\Gr(k,n)$ and $\Gr(n-k,n)$ remains an isometry under these distances. In particular, both \eqref{eq:ineq1} and \eqref{eq:ineq2} still hold with any of these distances in place of the Grassmann distance.
\end{proof}
We will see in Section~\ref{sec:existing} that the projection distance $\delta^\pi$ in Theorem~\ref{thm:othermetrics} is equivalent to the containment gap, a measure of distance between subspaces of different dimensions originally proposed in operator theory \cite{Kato}.
%
%We end this section with a remark about the complexity of computing $\delta^*$, which falls under the general problem of computing distance of a point to a subvariety in a Grassmannian $\Gr(k,n)$. For the special case of the Euclidean space $\mathbb{R}^n = \Gr(0,n)$, the problem often arise in applications \cite{EHOST, OSS}. Nonetheless there are abundant examples of simple varieties where the problem  is intractable: E.g., for a $3$-factor Segre variety, the problem is NP-hard in the Cook--Karp--Levin sense \cite{HL}; for general varieties, it is at least as hard as deciding Hilbert Nullstellensatz, which is NP-complete in the Blum--Shub--Smale sense \cite{BCSS, BSS}. Having a Grassmannian instead of a Euclidean space as the ambient space further complicates the problem since distances in Grassmannians require more effort to compute than Euclidean distance (i.e., $l^2$-norm).  It is therefore somewhat surprising that all the distances in Theorems~\ref{thm1} and \ref{thm:othermetrics} can be readily computed in polynomial time to any fixed accuracy via the \textsc{svd}.

%%%%
\section{Grassmannian of subspaces of all dimensions}\label{sec:dinfty}

We view the equality of $d_{\Gr(k,n)}\bigl(\mathbf{A},\Omega_{-}(\mathbf{B})\bigr)$ and $d_{\Gr(l,n)}\bigl(\mathbf{B},\Omega_{+}(\mathbf{A})\bigr)$ as the strongest evidence that their common value $\delta(\mathbf{A}, \mathbf{B})$ provides the most natural notion of distance between subspaces of different dimensions. As we pointed out earlier, $\delta$ is a distance in the sense of a distance from a point to a set, but not a distance in the sense of a metric on the set of all subspaces of all dimensions. For instance, $\delta$ does not satisfy the separation property: $\delta(\mathbf{A}, \mathbf{B}) =0$ for any $\mathbf{A} \subsetneq \mathbf{B}$. In fact, it is easy to observe the following.
\begin{lemma}\label{lem:zero}
Let $\mathbf{A} \in \Gr(k,n)$ and $\mathbf{B} \in \Gr(l,n)$. Then $\delta(\mathbf{A},\mathbf{B}) = 0$ iff $\mathbf{A} \subseteq \mathbf{B}$ or $\mathbf{B} \subseteq \mathbf{A}$.
\end{lemma}

$\delta$ also does not satisfy the triangle inequality: For a line $\mathbf{L}$ not contained in a subspace $\mathbf{A}$, the triangle inequality, if true, would imply
\begin{align*}
\delta(\mathbf{L},\mathbf{A})&=\delta(\mathbf{L},\mathbf{A})+\delta(\mathbf{A},\mathbf{B})\ge \delta(\mathbf{L},\mathbf{B}),\\
\delta(\mathbf{L},\mathbf{B})&=\delta(\mathbf{L},\mathbf{B})+\delta(\mathbf{A},\mathbf{B})\ge \delta(\mathbf{L},\mathbf{A}),
\end{align*}
giving $\delta(\mathbf{L},\mathbf{A})=\delta(\mathbf{L},\mathbf{B})$ for any subspace $\mathbf{B}$, which is evidently false by Lemma~\ref{lem:zero} (e.g.\ take $\mathbf{B} = \mathbf{A} \oplus  \mathbf{L}$).

These observations also apply verbatim to all the other similarly-defined distances $\delta^*$ in Theorem~\ref{thm:othermetrics}, i.e., none of them are metrics.

The set of all subspaces of all dimensions is parameterized by $\Gr(\infty, \infty)$, the \emph{doubly infinite Grassmannian}  \cite{FH}, which may be viewed informally as the disjoint union of all $k$-dimensional subspaces\footnote{As discussed in Section~\ref{sec:infty}, these are independent of the dimension of their ambient space and may be viewed as an element of the infinite Grassmannian $\Gr(k,\infty)$.} over all $k \in \mathbb{N}$,
\[
\Gr(\infty, \infty)=\coprod\nolimits_{k=1}^{\infty} \Gr(k,\infty).
\]
To define a metric between any pair of subspaces of arbitrary dimensions is to define one on $\Gr(\infty,\infty)$. It is easy to define metrics on $\Gr(\infty,\infty)$ that bear little relation to the geometry of Grassmannian but we will propose one in Section~\ref{sec:metric} that is consistent with  $\delta$ and with $d_{\Gr(k,n)}$ for all $k \le n$.

We will require the formal definition of $\Gr(\infty, \infty)$, namely, it is the direct limit of the direct system of Grassmannians $\{ \Gr(k,n) : (k,n) \in \mathbb{N} \times \mathbb{N} \}$ with inclusion maps $\iota^{kl}_{nm} : \Gr(k,n) \to \Gr(l,m)$ for all $k \le l$ and $n \le m$ such that $l-k\le m-n$.
For $A \in \mathbb{R}^{n \times k}$ with orthonormal columns, the embedding is given by
\begin{equation}\label{eq:epsilon}
\iota^{kl}_{nm} :\Gr(k,n) \to \Gr(l,m), \qquad
\operatorname{span} (A) \mapsto
\operatorname{span} \left(
\begin{bmatrix}
A & 0\\
0 & 0\\
0 &  I_{l-k} 
\end{bmatrix}\right), 
\end{equation}
where $I_{l-k} \in \mathbb{R}^{(l -k) \times (l -k)}$ is an identity matrix and we have $(m-n) - (l -k)$ zero rows in the middle so that the $3 \times 2$ block matrix is in  $\mathbb{R}^{m \times l}$. Note that for a fixed $k$, $\iota^{kk}_{nm} $ reduces to $\iota_{nm}$ in \eqref{eq:iota}.

Since our distance $\delta(\mathbf{A},\mathbf{B})$ is defined for subspaces $\mathbf{A}$ and $\mathbf{B}$ of all dimensions, it defines a function $\delta:\Gr(\infty,\infty)\times \Gr(\infty,\infty)\to \mathbb{R}$ that is a \emph{premetric} on $\Gr(\infty,\infty)$, i.e., $\delta(\mathbf{A},\mathbf{B})\ge 0$ and $\delta(\mathbf{A},\mathbf{A})=0$ for all $\mathbf{A},\mathbf{B} \in \Gr(\infty,\infty)$. This in turn defines a topology $\tau $ on $\Gr(\infty,\infty)$ in a standard way: The $\varepsilon$-ball centered at $\mathbf{A}$ is
\[
B_\varepsilon(\mathbf{A})\coloneqq\{ \mathbf{X}\in \Gr(\infty,\infty) : \delta(\mathbf{A},\mathbf{X})< \varepsilon \},
\]
and $U \subseteq \Gr(\infty,\infty)$ is defined to be open if for any $\mathbf{A}\in U$, there is an $\varepsilon$-ball $B_\varepsilon(\mathbf{A}) \subseteq U$. The topology $\tau$ is consistent with the usual topology of Grassmannians (but it is not the disjoint union topology). If we restrict $\tau$ to $\Gr(k,\infty)$, then the subspace topology is the same as the topology induced by the metric $d_{\Gr(k,\infty)}$ on $\Gr(k,\infty)$ as defined in Section~\ref{sec:infty}. Nevertheless this apparently natural topology on $\Gr(\infty,\infty)$ turns out to be a strange one.
\begin{proposition}\label{prop:tau}
The topology $\tau$ on $\Gr(\infty,\infty)$ is non-Hausdorff and therefore non-metrizable.
\end{proposition}
\begin{proof} $\tau $ is not Hausdorff since it is not possible to separate $\mathbf{A}\subsetneq \mathbf{B}$ by open subsets, as we saw in Lemma~\ref{lem:zero}. Metrizable spaces are necessarily Hausdorff.
\end{proof}

Even though $\tau$ restricts to the metric space topology on $\Gr(k,\infty)$ induced by the Grassmann distance $d_{\Gr(k,\infty)}$ for every $k \in \mathbb{N}$, it is not itself a metric space topology. We view this as a consequence of a more general phenomenon, namely, the category $\boldsymbol{\mathsf{Met}}$ of metric spaces (objects) and continuous contractions (morphisms) has no coproduct, i.e., given a collection of metric spaces, there is in general no metric space that will behave like the disjoint union of the collection of metric spaces. To see this, take metric spaces $(X_1,d_1)$ and $(X_2,d_2)$ where $X_1 = \{ x_1\}$, $X_2 = \{x_2\}$. Suppose a coproduct $(X,d)$ of $(X_1,d_1)$ and $(X_2,d_2)$ exists. Let $Y=\{y_1,y_2\}$ and let $d_Y$ be the metric on $Y$ induced by $d_Y(y_1,y_2)=2 d(x_1,x_2)\ne 0$. Now define $\varphi_i: X_i\to Y$ by $\varphi_i(x_i)=y_i$, $i=1,2$. One sees that no morphism $\varphi:X\to Y$ in $\boldsymbol{\mathsf{Met}}$ is compatible with $\varphi_1$ and $\varphi_2$, contradicting the assumption that $X$ is the coproduct of $X_1$ and $X_2$.

If we instead look at the category of metric spaces with continuous or uniformly continuous maps as morphisms, then coproducts always exist \cite{Helemskii}.  In Section~\ref{sec:metric}, we will relax our requirement and construct a metric $d_{\Gr(\infty,\infty)}$ on $\Gr(\infty,\infty)$ that restricts to $d_{\Gr(k,\infty)}$ for all $k \in \mathbb{N}$ but without requiring that it comes from a coproduct of $\{(\Gr(k,\infty), d_{\Gr(k,\infty)}) : k \in \mathbb{N}\}$ in $\boldsymbol{\mathsf{Met}}$. 

\section{Metrics for subspaces of all dimensions}\label{sec:metric}

We will describe a simple recipe for turning the distances $\delta^*$ in Theorem~\ref{thm:othermetrics} into metrics on $\Gr(\infty,\infty)$. 
Suppose $k \le l$ and we have $\mathbf{A} \in \Gr(k,n)$ and $\mathbf{B} \in \Gr(l,n)$. In this case there are $k$ principal angles between $\mathbf{A}$ and $\mathbf{B}$, $\theta_1,\dots, \theta_k$, as defined in \eqref{eq:angles}. First we will set
$\theta_{k+1} = \dots = \theta_l = \pi/2$.
Then we take the Grassmann distance $\delta$ or any of the distances $\delta^*$  in Theorem~\ref{thm:othermetrics}, replace the index $k$ by $l$, and call the resulting expressions $d_{\Gr(\infty, \infty)}(\mathbf{A}, \mathbf{B})$ (for Grassmann distance) and $d^*_{\Gr(\infty, \infty)}(\mathbf{A}, \mathbf{B})$ (for other distances) respectively.

When $n$ is sufficiently large, setting $\theta_{k+1},\dots,\theta_l$ all equal to $\pi/2$ is equivalent to completing $\mathbf{A}$ to an $l$ dimensional subspace of $\mathbb{R}^n$, by adding $l-k$ vectors orthonormal to the subspace $\mathbf{B}$. Hence the distance between $\mathbf{A}$ and $\mathbf{B}$ is defined by the distance function on the Grassmannian $\Gr(l,n)$. We show in Proposition~\ref{prop:metric} that these expressions will indeed define metrics on $\Gr(\infty,\infty)$.

Applying the above recipe to the Grassmann, chordal, and Procrustes distances yield the \emph{Grassmann}, \emph{chordal}, and \emph{Procrustes metrics} on $\Gr(\infty,\infty)$ given in Table~\ref{tab:metrics}.
\begin{table*}[ht]
\caption{Metrics on $\Gr(\infty,\infty)$ in terms of principal angles.}
\label{tab:metrics}
\centering
\vspace*{-1.5ex}
\renewcommand{\arraystretch}{1.75}
\begin{tabular}{ll}
Grassmann metric & $d_{\Gr(\infty, \infty)}(\mathbf{A}, \mathbf{B}) = \Bigl(\lvert k - l \rvert\pi^2/4 + \sum\nolimits_{i=1}^{\min(k,l)}\theta_i^2\Bigr)^{1/2}$\\
Chordal metric & $d_{\Gr(\infty, \infty)}^{\kappa}(\mathbf{A}, \mathbf{B}) = \Bigl(\lvert k - l \rvert + \sum\nolimits_{i=1}^{\min(k,l)} \sin^2\theta_i\Bigr)^{1/2}$ \\
Procrustes metric & $d_{\Gr(\infty, \infty)}^{\rho}(\mathbf{A}, \mathbf{B}) =\Bigl(\lvert k - l \rvert + 2\sum\nolimits_{i=1}^{\min(k,l)}\sin ^2(\theta_i/2)\Bigr)^{1/2}$
\end{tabular}
\end{table*}

Evidently  the metrics in Table~\ref{tab:metrics} are all of the form
\begin{equation}\label{eq:rms}
d^*_{\Gr(\infty,\infty)} (\mathbf{A},\mathbf{B}) = \sqrt{\delta^*(\mathbf{A}, \mathbf{B})^2 + c_*^2 \epsilon(\mathbf{A},\mathbf{B})^2},
\end{equation}
where $\epsilon(\mathbf{A},\mathbf{B}) \coloneqq \lvert \dim \mathbf{A} - \dim \mathbf{B} \rvert^{1/2}$.
On the other hand, applying the above recipe to other distances in Table~\ref{tab:distances} yield the \emph{Asimov}, \emph{Binet--Cauchy}, \emph{Fubini--Study}, \emph{Martin}, \emph{projection}, and \emph{spectral metrics} on $\Gr(\infty,\infty)$ given by
\begin{equation}\label{eq:indicator}
d_{\Gr(\infty, \infty)}^{*}(\mathbf{A}, \mathbf{B})  \\
 =
\begin{cases}
d_{\Gr(k,\infty)}^* (\mathbf{A}, \mathbf{B}) & \text{if } \dim \mathbf{A} = \dim \mathbf{B} = k,\\
c_* & \text{if } \dim \mathbf{A} \ne \dim \mathbf{B},
\end{cases}
\end{equation}
for  $ * = \alpha, \beta, \phi, \mu, \pi, \sigma$, respectively. The constants $c_* > 0$ can be seen to be
\[
c = c_\alpha = \pi/2, \quad  c_\sigma = \sqrt{2}, \quad c_\mu = \infty,\quad
 c_\beta = c_\phi = c_\pi = c_\kappa = c_\rho = 1.
\]
In all cases, for subspaces $\mathbf{A}$ and $\mathbf{B}$ of equal dimension $k$, these metrics on $\Gr(\infty,\infty)$ restrict to the corresponding ones on $\Gr(k,\infty)$, i.e.,
\[
d^*_{\Gr(\infty, \infty)}(\mathbf{A}, \mathbf{B}) = d^*_{\Gr(k, \infty)}(\mathbf{A}, \mathbf{B}),
\]
where the latter is as described in Corollary~\ref{cor:infty} and Lemma~\ref{lem:inclusion}. These metrics on  $\Gr(\infty,\infty)$ are the amalgamation of two pieces of information, the distance $\delta^*(\mathbf{A},\mathbf{B})$ and the difference in dimensions $\lvert \dim \mathbf{A} - \dim \mathbf{B} \rvert$, either via a root mean square or an indicator function.

The Grassmann metric has a natural interpretation (see  Proposition~\ref{prop:prob4}):
\begin{quote}
\emph{$d_{\Gr(\infty,\infty)} (\mathbf{A},\mathbf{B})$ is the distance  from $\mathbf{B}$ to the furthest $l$-dimensional subspace that contains $\mathbf{A}$, which equals the distance  from $\mathbf{A}$ to the furthest $k$-dimensional subspace contained in $\mathbf{B}$.}
\end{quote}
The chordal metric in Table~\ref{tab:metrics} is equivalent to the symmetric directional distance, a metric on subspaces of different dimensions  \cite{SWF, WWF}  popular in machine learning  \cite{bar, ref2a, motion, ref2b, eeg, mech, blog, ref2c, econ, network, face, ref2d, text} (see Section~\ref{sec:existing}). 

\begin{proposition}\label{prop:metric}
The expressions in Table~\ref{tab:metrics} and \eqref{eq:indicator} are metrics on $\Gr(\infty,\infty)$.
\end{proposition}
\begin{proof}
It is trivial to see that the expression defined in \eqref{eq:indicator} yields a metric on $\Gr(\infty,\infty)$ for $ * = \alpha, \beta, \mu, \pi, \sigma, \phi$, and so we just need to check the remaining three cases that take the form in \eqref{eq:rms}. Of the four defining properties of a metric, only the triangle inequality is not immediately clear from \eqref{eq:rms}.

Let $k = \dim \mathbf{A}$, $l = \dim \mathbf{B}$, and $m = \dim \mathbf{C}$. We may assume \textsc{wlog} that $k\le l \le m \le n$ where $n$ is chosen sufficiently large so that $\mathbf{A}, \mathbf{B}, \mathbf{C}$ are subspaces in $\mathbb{R}^n$.  Let $A \in \mathbb{R}^{n\times k}$, $B \in \mathbb{R}^{n \times l}$, $C \in \mathbb{R}^{n \times m}$ be matrices whose columns are orthonormal bases of $\mathbf{A}$, $\mathbf{B}$, $\mathbf{C}$ respectively. Consider the following $(n+m-k)\times m$ matrices:
\[
A'=\begin{bmatrix}
A & 0\\
0 & I_{m-k}
\end{bmatrix},\quad
B'=\begin{bmatrix}
B & 0\\
0 & 0\\
0 & I_{m-l}
\end{bmatrix},\quad
C'=\begin{bmatrix}
C\\
0
\end{bmatrix}.
\]
and set $\mathbf{A}' = \operatorname{span}(A')$, $\mathbf{B}' = \operatorname{span}(B')$, $\mathbf{C}' = \operatorname{span}(C')$; note that these are just $\mathbf{A}$, $\mathbf{B}$, $\mathbf{C}$ embedded in $\Graff(m,n+m-k)$ via \eqref{eq:epsilon}. The expressions in Table~\ref{tab:metrics} satisfy
\begin{align*}
d^*_{\Gr(\infty,\infty)} (\mathbf{A},\mathbf{B}) = d^*_{\Gr(m,n+m-k)}(\mathbf{A}',\mathbf{B}'),\\
d^*_{\Gr(\infty,\infty)} (\mathbf{B},\mathbf{C}) = d^*_{\Gr(m,n+m-k)}(\mathbf{B}',\mathbf{C}'),\\
d^*_{\Gr(\infty,\infty)} (\mathbf{A},\mathbf{C}) = d^*_{\Gr(m,n+m-k)}(\mathbf{A}',\mathbf{C}').
\end{align*}
Since $\mathbf{A}',\mathbf{B}',\mathbf{C}' \in \Gr(m,n+m-k)$, the triangle inequality for $d^*_{\Gr(m,n+m-k)}$ immediately
%\begin{multline*}
%d^*_{\Gr(m,n+m-k)}(\mathbf{A}',\mathbf{C}') \le d^*_{\Gr(m,n+m-k)}(\mathbf{A}',\mathbf{B}')\\
%+d^*_{\Gr(m,n+m-k)}(\mathbf{B}',\mathbf{C}'),
%\end{multline*}
yields the triangle inequality for $d^*_{\Gr(\infty,\infty)}$.
\end{proof}
The proof shows that for any $\mathbf{A} \in \Gr(k,n)$ and $\mathbf{B} \in \Gr(l,n)$ where $k \le l \le n$,
\[
d^*_{\Gr(\infty,\infty)}(\mathbf{A},\mathbf{B})
=d^*_{\Gr(l,n+l-k)}\bigl(\iota^{k,l}_{n,n+l-k}(\mathbf{A}),\iota^{l,l}_{n,n+l-k}(\mathbf{B})\bigr).
\]
The embeddings $\iota^{k,l}_{n,n+l-k}:\Gr(k,n)\to \Gr(l,n+l-k)$ and $\iota^{l,l}_{n,n+l-k}:\Gr(l,n)\to \Gr(l,n+l-k)$ are as defined in \eqref{eq:epsilon} and are isometric for all $ k\le l \le n$.

\begin{proposition}\label{prop:prob4}
Let  $k\le l \le n/2$ and $\mathbf{A} \in \Gr(k,n)$, $\mathbf{B}\in \Gr(l,n)$. Then
\begin{equation}\label{eq:equal2}
\max_{\mathbf{X}\in \Omega_{+}(\mathbf{A})}d_{\Gr(l,n)}(\mathbf{X},\mathbf{B})  = 
\max_{\mathbf{Y} \in \Omega_{-}(\mathbf{B})} d_{\Gr(k,n)}(\mathbf{Y},\mathbf{A})
= d_{\Gr(\infty,\infty)} (\mathbf{A},\mathbf{B}),
\end{equation}
i.e., $d_{\Gr(\infty,\infty)}$ is the distance between \emph{furthest} subspaces. 
\end{proposition}
\begin{proof}
We assume \textsc{wlog} that $\mathbf{A}\cap \mathbf{B}=\{0 \}$ by Proposition~\ref{prop:angles}.  Since
\[
d_{\Gr(l,n)}(\mathbf{X},\mathbf{B})=\delta(\mathbf{X},\mathbf{B})=\Bigl(\sum\nolimits_{i=1}^l \theta_i(\mathbf{X},\mathbf{B})^2\Bigr)^{1/2},
\]
and by Corollary~\ref{cor:compare},
$\theta_i(\mathbf{X},\mathbf{B})\le \theta_i(\mathbf{A},\mathbf{B})$, $i =1,\dots,k$,
we obtain
\[%\label{eq3}
d_{\Gr(l,n)}(\mathbf{X},\mathbf{B})\le \Bigl(\delta(\mathbf{A},\mathbf{B})^2+\sum\nolimits_{i=k+1}^l\theta_i(\mathbf{X},\mathbf{B})^2\Bigr)^{1/2}.
\]
Let $(a_1,b_1),\dots, (a_k, b_k)$ be the principal vectors between $\mathbf{A}$ and $\mathbf{B}$. We extend $b_1,\dots, b_k$ to obtain an orthonormal basis $b_1,\dots, b_k,b_{k+1},\dots, b_l$ of $\mathbf{B}$. Let $\mathbf{X}\cap\mathbf{A}^\perp$ be the orthogonal complement of $\mathbf{A}$ in $\mathbf{X}$ and let $\mathbf{B}_0 \coloneqq\operatorname{span}\{b_{k+1},\dots, b_l\}$. Then
\[
\Bigl(\sum\nolimits_{i=k+1}^l\theta_i(\mathbf{X},\mathbf{B})^2 \Bigr)^{1/2} = \delta(\mathbf{X}\cap\mathbf{A}^\perp,\mathbf{B}_0),
\]
and the last inequality becomes
\[
d_{\Gr(l,n)}(\mathbf{X},\mathbf{B}) \le  \sqrt{\delta(\mathbf{A},\mathbf{B})^2+\delta(\mathbf{X}\cap\mathbf{A}^\perp,\mathbf{B}_0)^2}.
\]
If $n\ge 2l$, then there exist $l-k$ vectors $c_1,\dots,c_{l-k}$ orthogonal to $\mathbf{A}$ and $\mathbf{B}$ simultaneously. Choosing $\mathbf{X} = \operatorname{span}\{a_1,\dots,a_k,c_1,\dots,c_{l-k}\}$, we attain the required maximum:
\[
d_{\Gr(l,n)}(\mathbf{X},\mathbf{B})=\sqrt{\delta(\mathbf{A},\mathbf{B})^2+(l-k)\pi^2/4} = d_{\Gr(\infty,\infty)} (\mathbf{A},\mathbf{B}).
\]
The second equality in \eqref{eq:equal2} follows from $d_{\Gr(\infty,\infty)} (\mathbf{A},\mathbf{B}) = d_{\Gr(\infty,\infty)} (\mathbf{B},\mathbf{A})$, given that $d_{\Gr(\infty,\infty)}$ is a metric by Proposition~\ref{prop:metric}.
\end{proof}

The existence of the metrics $d^*_{\Gr(\infty,\infty)}$ as defined in \eqref{eq:rms} and \eqref{eq:indicator}  does not contradict our earlier discussion about the general nonexistence of coproduct in $\boldsymbol{\mathsf{Met}}$ as these metrics do not respect continuous \emph{contractions}.
Take the Grassmann metric on $\Gr(\infty,\infty)$ for instance. $(\Gr(\infty,\infty),d_{\Gr(\infty,\infty)})$ is an object of the category $\boldsymbol{\mathsf{Met}}$  but it is \emph{not} the coproduct of  $\{(\Gr(k,\infty), d_{\Gr(k,\infty)} ) : k \in \mathbb{N}\}$. Indeed, let $Y = \{y_1, y_2\}$ with metric defined by $d_Y(y_1, y_2) = 1$. Consider a family of maps $f_k:\Gr(k,\infty)\to Y$,
\[ 
f_k(\mathbf{A})=
\begin{cases}
y_1&\text{if }k=2,\\
y_2&\text{otherwise}.
\end{cases}
\]
Then $f_k$ is a continuous contraction between $\Gr(k,\infty)$ and $Y$. So $\{f_k :k \in \mathbb{N}\}$ is a family of morphisms in $\boldsymbol{\mathsf{Met}}$ compatible with $\{(\Gr(k,\infty), d_{\Gr(k,\infty)} ) : k \in \mathbb{N}\}$. If $(\Gr(\infty,\infty),d_{\Gr(\infty,\infty)})$ is the coproduct of this family, then there must be a continuous contraction $f:\Gr(\infty,\infty)\to Y$ such that $f\circ \iota_k=f_k$ with $\iota_k$ being the natural inclusion of $\Gr(k,\infty)$ into $\Gr(\infty,\infty)$. But taking $\mathbf{A}\in \Gr(2,\infty)$ and $\mathbf{B}\in \Gr(3,\infty)$, we see that
\[
d_{\Gr(\infty,\infty)}(\mathbf{A},\mathbf{B})\ge \frac{\pi}{2} > 1=d_Y\bigl(f(\mathbf{A}),f(\mathbf{B})\bigr),
\]
contradicting the surmise that $f$ is a contraction. Similarly, one may show that $(\Gr(\infty,\infty),d_{\Gr(\infty,\infty)}^*)$ is not a coproduct in $\boldsymbol{\mathsf{Met}}$ for any $* = \alpha, \beta, \kappa, \mu, \pi, \rho, \sigma, \phi$.

$(\Gr(\infty,\infty),d_{\Gr(\infty,\infty)})$ is also not the coproduct of $\{(\Gr(k,\infty), d_{\Gr(k,\infty)} ) : k \in \mathbb{N}\}$ in the category of metric spaces with continuous (or uniformly continuous) maps as morphisms. The coproduct in this category is simply $\Gr(\infty,\infty)$ with the metric induced by the disjoint union topology, which is too fine (in the sense of topology) to be interesting. In particular, such a metric is unrelated to the distance $\delta$.

%%%%
\section{Comparison with existing works}\label{sec:existing}

There are two existing proposals for a distance between subspaces of different dimensions --- the containment gap and the symmetric directional distance. These turn out to be special cases of our distance in Section~\ref{sec:main} and our metric in Section~\ref{sec:metric}. 

Let $\mathbf{A} \in \Gr(k,n)$ and $\mathbf{B} \in \Gr(l,n)$.  The \emph{containment gap} is defined as 
\[
\gamma(\mathbf{A}, \mathbf{B}) \coloneqq \max_{a\in \mathbf{A}} \; \min_{b\in \mathbf{B}} \frac{\lVert a - b \rVert}{\lVert a \rVert}.
\]
This was proposed in \cite[pp.~197--199]{Kato} and used in numerical linear algebra \cite{SS} for measuring separation between Krylov subspaces \cite{BES}. It is equivalent to our \emph{projection distance} $\delta^\pi$ in  Theorem~\ref{thm:othermetrics}. It was observed in \cite[p.~495]{BES} that
\[
\gamma(\mathbf{A}, \mathbf{B}) = \sin\bigl( \theta_k(\mathbf{A},\mathbf{Y}) \bigr)
\]
where $\mathbf{Y} \in \Omega_{-}(\mathbf{B})$ is nearest to $\mathbf{A}$ in the projection distance $d^\pi_{\Gr(k,n)}$. By Theorem~\ref{thm:othermetrics}, we deduce that it can also be realized as
\[
\gamma(\mathbf{A}, \mathbf{B}) = \sin\bigl( \theta_l(\mathbf{B},\mathbf{X}) \bigr)
\]
where $\mathbf{X} \in \Omega_{+}(\mathbf{A})$ is nearest to $\mathbf{B}$ in the projection distance $d^\pi_{\Gr(l,n)}$, a fact about the containment gap that had not been observed before. Indeed, by Theorem~\ref{thm:othermetrics}, we get
\[
\gamma(\mathbf{A}, \mathbf{B}) = \delta^\pi (\mathbf{A}, \mathbf{B})
\]
for all $\mathbf{A} \in \Gr(k,n)$ and $\mathbf{B} \in \Gr(l,n)$.

The \emph{symmetric directional distance} is defined as
\begin{equation}\label{eq:sdd}
d_{\Delta}(\mathbf{A},\mathbf{B})  \coloneqq \Bigl( \max(k,l) - \sum\nolimits_{i,j=1}^{k,l} (a_i^\mathsf{T} b_j)^2 \Bigr)^{1/2}
\end{equation}
where $A = [a_1,\dots,a_k]$ and $B = [b_1,\dots,b_l]$ are the respective orthonormal bases. This was proposed in \cite{SWF, WWF}, and is widely used  \cite{bar, ref2a, motion, ref2b, eeg, mech, blog, ref2c, econ, network, face, ref2d, text}. The definition \eqref{eq:sdd} is equivalent to our \emph{chordal metric} $d^\kappa_{\Gr(\infty,\infty)}$ in Table~\ref{tab:metrics},
\[
d_{\Gr(\infty,\infty)}^\kappa (\mathbf{A},\mathbf{B})^2 = \lvert k - l \rvert + \sum_{i=1}^{\min(k,l)} \sin^2 \theta_i = \max(k,l) - \sum_{i,j=1}^{k,l} (a_i^\mathsf{T} b_j)^{2} =d_\Delta (\mathbf{A}, \mathbf{B})^2,
\]
since $ \lvert k - l \rvert  = \max(k,l) - \min(k,l)$, and
\[
 \sum\nolimits_{i,j=1}^{k,l} (a_i^\mathsf{T} b_j)^{2} = \lVert A^\mathsf{T} B \rVert_F^2 
= \sum\nolimits_{i=1}^{\min(k,l)}  \cos^2 \theta_i  = \min(k,l) - \sum\nolimits_{i=1}^{\min(k,l)}  \sin^2 \theta_i .
\]

%%%%
\section{Geometry of $\Omega_{+}(\mathbf{A})$ and $\Omega_{-}(\mathbf{B})$}\label{sec:schubert}

Up to this point, $\Omega_{+}(\mathbf{A})$ and $\Omega_{-}(\mathbf{B})$, as defined in Definition~\ref{def:Omega}, are treated as mere subsets of $\Gr(l,n)$ and $\Gr(k,n)$ respectively. We will see that $\Omega_{+}(\mathbf{A})$ and $\Omega_{-}(\mathbf{B})$ have rich geometric properties. Firstly, we will show that they are Schubert varieties, justifying their names.
\begin{definition}\label{def:Schubert}
Let $\mathbf{X}_1\subset \mathbf{X}_2 \subset \dots \subset \mathbf{X}_k$ be a fixed flag in $\mathbb{R}^n$. The \emph{Schubert variety} $\Omega(\mathbf{X}_1, \dots, \mathbf{X}_k , n)$ is the set of $k$-planes $\mathbf{Y}$ satisfying the Schubert conditions
$\dim (\mathbf{Y}\cap \mathbf{X}_i)\ge i$, $i=1,\dots, k$, i.e.,
\[
\Omega(\mathbf{X}_1, \dots, \mathbf{X}_k , n) = \{\mathbf{Y} \in \Gr(k,n):  \dim (\mathbf{Y}\cap \mathbf{X}_i)\ge i, \; i=1,\dots, k \}.
\]
\end{definition}
\begin{definition}
Let $0 \eqqcolon k_0 < k_1<\cdots < k_{m+1} \coloneqq n$ be a sequence of increasing nonnegative integers. The associated \emph{flag variety}  is the set of flags satisfying the condition
$\dim \mathbf{X}_i = k_i$, $i=0, 1,\dots, m+1$.
We denote it by $\Flag(k_1,\dots, k_m, n)$, i.e.,
\[
\{(\mathbf{X}_1,\dots,\mathbf{X}_{m}) \in \Gr(k_1,n) \times \dots \times \Gr(k_m,n) : 
 \mathbf{X}_i\subset \mathbf{X}_{i+1}, \; i =1,\dots,m \}.
\]
\end{definition}
Observe that a Schubert variety depends on a specific increasing sequence of subspaces whereas a flag variety depends only on an increasing sequence of dimensions (of subspaces). Flag varieties may be viewed as a generalization of Grassmannians since if $m=1$, then $\Flag(k,n) = \Gr(k,n)$. Like Grassmannians, $\Flag(k_1,\dots, k_m, n)$ is a smooth manifold and sometimes called a \emph{flag manifold}. The parallel goes further, $\Flag(k_1,\dots, k_m, n) $ is a homogeneous space,
\begin{equation}\label{eq:flag}
\Flag(k_1,\dots, k_m, n) \cong \O(n)/\bigl(\O(d_1) \times \dots \times \O(d_{m+1})\bigr)
\end{equation}
where $d_i = k_{i} - k_{i-1}$ for $i =1,\dots, m+1$, generalizing \eqref{eq:homo}.

Let $\mathbf{A}\in \Gr(k,n)$ and $\mathbf{B} \in \Gr(l,n)$ with $k\le l$. Then
\[
\Omega_{+}(\mathbf{A}) = \Omega(\mathbf{A}_1, \dots, \mathbf{A}_l , n),\qquad
\Omega_{-}(\mathbf{B}) = \Omega(\mathbf{B}_1, \dots, \mathbf{B}_{k} , n),
\]
are Schubert varieties in $\Gr(l,n)$ and  $\Gr(k,n)$ respectively with the flags
\begin{gather*}
\{0\} \eqqcolon \mathbf{A}_0 \subset \mathbf{A}_1\subset \dots \subset \mathbf{A}_k\coloneqq\mathbf{A}\subset \mathbf{A}_{k+1}\dots \subset \mathbf{A}_{l}, \\
\{0\} \eqqcolon \mathbf{B}_0 \subset \mathbf{B}_1 \subset  \dots \subset \mathbf{B}_{k} \coloneqq\mathbf{B}.
\end{gather*}
where $\mathbf{A}_{k+i}$ is a subspace of $\mathbb{R}^n$ containing $\mathbf{A}$ of dimension $n-l+(k+i)$ for $1\le i\le l-k$.

The isomorphism $\Gr(l,n) \cong \Gr(n-l,n)$ (resp.\ $\Gr(k,n) \cong \Gr(n-k,n)$) that sends $\mathbf{X}$ to $\mathbf{X}^{\perp}$ takes $\Omega_{+}(\mathbf{A})$ to $\Omega_{-}(\mathbf{A}^{\perp})$ (resp.\ $\Omega_{-}(\mathbf{B})$ to $\Omega_{+}(\mathbf{B}^{\perp})$). Thus $\Omega_{+}(\mathbf{A})$ (resp.\ $\Omega_{-}(\mathbf{B})$) may also be viewed as Schubert varieties in $\Gr(n-l,n)$ (resp.\ $\Gr(n-k,n)$).  More importantly, this observation implies that $\Omega_{+}(\mathbf{A})$ and $\Omega_{-}(\mathbf{B})$, despite superficial difference in their definitions, are essentially the same type of objects.

\begin{proposition}\label{prop:Omega}
For any $\mathbf{A} \in \Gr(k,n)$  and $\mathbf{B} \in \Gr(l,n)$, we have
\[
\Omega_{+}(\mathbf{A}) \cong \Omega_{-}(\mathbf{A}^{\perp}) \quad\text{and}\quad \Omega_{-}(\mathbf{B}) \cong \Omega_{+}(\mathbf{B}^{\perp}).
\]
\end{proposition}
Also, $\Omega_{+}(\mathbf{A})$ and  $\Omega_{-}(\mathbf{B})$ are uniquely determined by $\mathbf{A}$ and $\mathbf{B}$ respectively.
\begin{proposition}\label{prop:Omega1}
Let $\mathbf{A}, \mathbf{A}' \in \Gr(k,n)$ and  $\mathbf{B},\mathbf{B}' \in \Gr(l,n)$. Then
\begin{gather*}
\Omega_{+}(\mathbf{A}) =\Omega_{+}(\mathbf{A'})\quad \text{if and only if} \quad \mathbf{A} = \mathbf{A}',\\
\Omega_{-}(\mathbf{B}) =\Omega_{-}(\mathbf{B}')\quad \text{if and only if} \quad \mathbf{B} = \mathbf{B}'.
\end{gather*}
\end{proposition}
\begin{proof}
Suppose $\Omega_{+}(\mathbf{A})=\Omega_{+}(\mathbf{A'})$. Observe that the intersection of all $l$-planes containing $\mathbf{A}$ is exactly $\mathbf{A}$ and ditto for $\mathbf{A}'$. So
\[
\mathbf{A} = \bigcap\nolimits_{\mathbf{X} \in\Omega_{+}(\mathbf{A})} \mathbf{X} = \bigcap\nolimits_{\mathbf{X} \in\Omega_{+}(\mathbf{A}')} \mathbf{X} = \mathbf{A}'.
\]
The converse is obvious. The statement for $\Omega_{-}$ then follows from Proposition~\ref{prop:Omega}.
\end{proof}
This observation allows us to treat subspaces of different dimensions on the same footing by regarding them as \emph{subsets} in the same Grassmannian. If we have a collection of subspaces of dimensions $k\le k_1 < k_2 < \dots < k_m \le l$, the injective map $\mathbf{A} \mapsto \Omega_{+}(\mathbf{A})$ takes all of them into distinct subsets of $\Gr(l,n)$. Alternatively, the injective map  $\mathbf{B} \mapsto \Omega_{-}(\mathbf{B})$ takes all of them into distinct subsets of $\Gr(k,n)$.

The resemblance between $\Omega_{+}(\mathbf{A})$ and $\Omega_{-}(\mathbf{B})$ in Proposition~\ref{prop:Omega} goes further --- we may view them as `sub-Grassmannians'. 
\begin{proposition}\label{prop:OmegaGrass}
Let $k \le l \le n$ and $\mathbf{A} \in \Gr(k,n)$, $\mathbf{B} \in \Gr(l,n)$. Then
\[
\Omega_{+}(\mathbf{A}) \cong \Gr(l-k,n-k), \quad \Omega_{-}(\mathbf{B}) \cong \Gr(k,l),
\]
isomorphic as algebraic varieties and diffeomorphic as smooth manifolds. Thus
\[
\dim \Omega_{+}(\mathbf{A})=(n-l)(l-k),\quad \dim \Omega_{-}(\mathbf{B}) = k(l-k).
\]
\end{proposition}
\begin{proof}
The first isomorphism is the quotient map
$\varphi: \Omega_{+}(\mathbf{A})\to  \Gr_{l-k}(\mathbb{R}^n/\mathbf{A})$, $\mathbf{X} \mapsto \mathbf{X}/\mathbf{A} \subseteq \mathbb{R}^n/\mathbf{A}$,
composed with the isomorphism $ \Gr_{l-k}(\mathbb{R}^n/\mathbf{A}) \cong \Gr(l-k,n-k)$.
The second isomorphism is obtained by regarding a $k$-dimensional subspace $\mathbf{Y}$ of $\mathbb{R}^n$ in $\Omega_{-}(\mathbf{B}) $ as a $k$-dimensional subspace of $\mathbf{B}$, i.e.,
$\Omega_{-}(\mathbf{B}) = \Gr_k(\mathbf{B}) \cong \Gr(k,l)$.
\end{proof}

That $\Omega_{+}(\mathbf{A})$ and $\Omega_{-}(\mathbf{B})$ are Grassmannians allows us to infer the following:
\begin{enumerate}[\upshape (i)]
\item\label{top} as topological spaces, they are compact and path-connected;
\item as algebraic varieties, they are irreducible and nonsingular;
\item\label{diff} as differential manifolds, they are smooth and  any two points on them can be connected by a length-minimizing geodesic.
\end{enumerate}
The topology in \eqref{top} refers to the metric space topology, not Zariski topology. A consequence of compactness is that the distance $d_{\Gr(k,n)}\bigl(\mathbf{A},\Omega_{-}(\mathbf{B})\bigr) = d_{\Gr(l,n)}\bigl(\mathbf{B},\Omega_{+}(\mathbf{A})\bigr)$ can be attained by points in $\Omega_{-}(\mathbf{B})$ and $\Omega_{+}(\mathbf{A})$ respectively. We constructed these closest points explicitly when we proved Theorem~\ref{thm1}.

Many more topological and geometric properties of $\Omega_{+}(\mathbf{A})$ and $\Omega_{-}(\mathbf{B})$ follow from Proposition~\ref{prop:OmegaGrass} as they inherit everything that we know about Grassmannians (e.g.\ coordinate ring, cohomology ring, Pl\"ucker relations, etc.); in particular, $\Omega_{+}(\mathbf{A})$ and $\Omega_{-}(\mathbf{B})$ are also flag varieties.

The last property in \eqref{diff} requires a proof. The length-minimizing  geodesic is not unique and so $\Omega_+(\mathbf{A})$ and $\Omega_-(\mathbf{B})$ are not \emph{geodesically convex} \cite[Definition~4.1.35]{Nicolaescu}.
\begin{proposition}
Any two points in $\Omega_{-}(\mathbf{B})$ (resp.\ $\Omega_{+}(\mathbf{A})$) can be connected by a length-minimizing geodesic in $\Gr(k,n)$ (resp.\ $\Gr(l,n)$).
\end{proposition}
\begin{proof}
By Proposition~\ref{prop:Omega}, it suffices to show that any two points in $\Omega_{-}(\mathbf{B})$ can be connected by a geodesic curve in $\Omega_{-}(\mathbf{B})$. By Proposition~\ref{prop:OmegaGrass}, $\Omega_{-}(\mathbf{B})$ is the image of $\Gr(k,l)$ embedded isometrically in $\Gr(k,n)$. So by Lemma~\ref{lem:infty}, for any $\mathbf{X}_1,\mathbf{X}_2 \in \Gr(k,l)$,
$d_{\Gr(k,n)}(\mathbf{X}_1,\mathbf{X}_2)= d_{\Gr(k,l)}(\mathbf{X}_1,\mathbf{X}_2)
= d_{\Omega_{-}(\mathbf{B})}(\mathbf{X}_1,\mathbf{X}_2)$,
where the last is the geodesic distance in $\Omega_{-}(\mathbf{B})$.
Hence if $d_{\Omega_{-}(\mathbf{B})}(\mathbf{X}_1,\mathbf{X}_2)$ is realized by a geodesic curve $\gamma$ in $\Omega_{-}(\mathbf{B})$, then $\gamma$ must also be a geodesic curve in $\Gr(k,n)$.
\end{proof}

%%%%
%\section{Grassmannians as matrix varieties}\label{sec:matrix}

We have represented $\Gr(k,n)$ as a set of \emph{equivalence classes} of matrices but it may also be represented as a set of \emph{actual matrices} \cite[Example~1.2.20]{Nicolaescu}, namely, idempotent symmetric matrices of trace $k$:
\[
\Gr(k,n)\cong \{ P \in \mathbb{R}^{n \times n} : P^{\mathsf{T}} = P^2 = P, \; \tr(P)=k \}.
\]
The isomorphism maps each subspace $\mathbf{A}\in \Gr(k,n)$ to $P_{\mathbf{A}} \in \mathbb{R}^{n \times n}$,  the unique orthogonal projection onto $\mathbf{A}$, and its inverse takes an orthogonal projection $P$ to the subspace $\im( P) \in \Gr(k,n)$.
$P$ is an orthogonal projection iff it is symmetric and idempotent, i.e., $P^{\mathsf{T}} = P^2 = P$. The eigenvalues of an orthogonal projection onto a subspace of dimension $k$ are $1$'s and $0$'s with multiplicities $k$ and $n-k$, so $\tr(P) = k$ is equivalent to $\Rank(P) = k$, ensuring $\im(P)$ has dimension $k$. In this representation,
\begin{align*}
\Omega_{+}(\mathbf{A}) &\cong \{ P \in \mathbb{R}^{n \times n} : P^{\mathsf{T}} = P^2 = P, \; \tr(P)=l, \; \im(A) \subseteq \im(P) \},\\
\Omega_{-}(\mathbf{B}) &\cong \{ P \in \mathbb{R}^{n \times n} : P^{\mathsf{T}} = P^2 = P, \; \tr(P)=k, \; \im(P) \subseteq \im(B) \},
\end{align*}
allowing us to treat $\Gr(k,n)$, $\Gr(l,n)$, $\Omega_{+}(\mathbf{A})$, $\Omega_{-}(\mathbf{B})$ all as subvarieties of $\mathbb{R}^{n \times n}$.

%%%%
\section{Probability density on the Grassmannian}\label{sec:volume}

We determine the relative volumes of $\Omega_{+}(\mathbf{A})$, $\Omega_{-}(\mathbf{B})$ and prove a volumetric analogue of \eqref{eq:equal} in Theorem~\ref{thm1}:
\begin{quote}
\emph{Given $k$-dimensional subspace $\mathbf{A}$ and $l$-dimensional subspace $\mathbf{B}$ in $\mathbb{R}^n$, the probability that a randomly chosen $l$-dimensional subspace in $\mathbb{R}^n$ contains $\mathbf{A}$ equals the probability that a randomly chosen $k$-dimensional subspace in $\mathbb{R}^n$ is contained in $\mathbf{B}$.}
\end{quote}
%The value of this probability is independent of $\mathbf{A}$ and $\mathbf{B}$ and only depends on $k,l,n$.

Every Riemannian metric on a Riemannian manifold yields a \emph{volume density} that is consistent with the metric \cite[Example~3.4.2]{Nicolaescu}. The Riemannian metric\footnote{Discussed at length in \cite{AMS, EAS}; we did not specify it since we have no use for it except implicitly.} on $\Gr(k,n)$ that gives us the Grassmann distance in \eqref{eq:grassdist1} and the geodesic in \eqref{eq:geodesic} also gives a density $d\gamma_{k,n}$ on $\Gr(k,n)$.  The volume of $\Gr(k,n)$ is then
\begin{equation}\label{eq:volume}
\Vol\bigl(\Gr(k,n)\bigr)=\int_{\Gr(k,n)} |d\gamma_{k,n}| =\binom{n}{k}\frac{\prod_{j=1}^{n}{\omega_j}}{\bigl(\prod_{j=1}^k\omega_j \bigr)\bigl(\prod_{j=1}^{n-k}\omega_j\bigr)},
\end{equation}
where $\omega_m \coloneqq \pi^{m/2}/\Gamma(1+m/2)$, volume of the unit ball in $\mathbb{R}^m$ \cite[Proposition~9.1.12]{Nicolaescu}.

The normalized density $d\mu_{k,n} \coloneqq \Vol\bigl(\Gr(k,n)\bigr)^{-1} \lvert d\gamma_{k,n}\rvert$ defines a natural \emph{uniform probability density} on $\Gr(k,n)$.  With respect to this, the probability of landing on $\Omega_{+}(\mathbf{A})$ in $\Gr(l,n)$ equals the probability of landing on $\Omega_{-}(\mathbf{B})$ in $\Gr(k,n)$.
\begin{corollary}\label{cor:volume}
Let $k \le l \le n$ and $\mathbf{A} \in \Gr(k,n)$, $\mathbf{B} \in \Gr(l,n)$. The relative volumes of $\Omega_{+}(\mathbf{A})$  in $\Gr(l,n)$  and $\Omega_{-}(\mathbf{B})$ in $\Gr(k,n)$ are equal and their common value  depends only on $k,l,n$,
\[
\mu_{l,n}\bigl(\Omega_{+}(\mathbf{A})\bigr) =
\mu_{k,n}\bigl(\Omega_{-}(\mathbf{B})\bigr) =
\frac{l!(n-k)!\prod_{j=l-k+1}^{l}{\omega_j}}{n!(l-k)!\prod_{j=n-k+1}^{n}\omega_j}.
\]
\end{corollary}
\begin{proof}
By Proposition~\ref{prop:OmegaGrass}, $\Omega_{+}(\mathbf{A})$ is isometric to $\Gr(n-l,n-k)$ and $\Omega_{-}(\mathbf{B})$ is isometric to $\Gr(k,l)$, so by \eqref{eq:volume} their volumes are
\[
\binom{n-k}{n-l}\frac{\prod_{j=1}^{n-k}{\omega_j}}{\bigl(\prod_{j=1}^{n-l}\omega_j\bigr)\bigl(\prod_{j=1}^{l-k}\omega_j)},
\qquad
\binom{l}{k}\frac{\prod_{j=1}^{l}{\omega_j}}{\bigl(\prod_{j=1}^k\omega_j\bigr)\bigl(\prod_{j=1}^{l-k}\omega_j)}
\]
respectively.  Now divide by the volumes of $\Gr(l,n)$ and $\Gr(k,n)$ respectively.
\end{proof}
By definition, relative volume depends on the volume of ambient space and the dependence on $n$ is expected, a slight departure from Theorem~\ref{thm1}\eqref{indp}.

%%%%
\section{Conclusions}

We provided what we hope is a thorough study of subspace distances, a topic of wide-ranging interest. We investigated the topic from different angles and filled in the most glaring gap in our existing knowledge --- defining distances and metrics for inequidimensional subspaces. We also developed simple geometric models for subspaces of all dimensions and enriched the existing differential geometric view of Grassmannians with algebraic geometric perspectives. We expect these to be of independent interest to applied and computational mathematicians. Most of the topics discussed in this article have been extended to affine subspaces in \cite{LWY}.

%%%%
\subsection*{Acknowledgment}

We are very grateful to the two anonymous referees for their invaluable suggestions, both mathematical and stylistic. We thank Sayan~Mukherjee for telling us about the importance of measuring distances between inequidimensional subspaces, Frank~Sottile for invaluable  discussions that led to the Schubert variety approach, and Lizhen~Lin, Tom~Luo, Giorgio~Ottaviani for helpful comments.

\bibliographystyle{plain}

\begin{thebibliography}{1}
\bibitem{AMSbook} P.-A.~Absil, R.~Mahony, and R.~Sepulchre, \emph{Optimization Algorithms on Matrix Manifolds}, Princeton University Press, Princeton, NJ, 2008.

\bibitem{AMS} P.-A.~Absil, R.~Mahony, and R.~Sepulchre, ``Riemannian geometry of Grassmann manifolds with a view on algorithmic computation,'' \emph{Acta Appl.\ Math.}, \textbf{80} (2004), no.~2, pp.~199--220. 

\bibitem{AMSV}  P.-A.~Absil, R.~Mahony, R.~Sepulchre, and P.~Van~Dooren, ``A Grassmann--Rayleigh quotient iteration for computing invariant subspaces,'' \emph{SIAM Rev.}, \textbf{44} (2002), no.~1, pp.~57--73. 

%\bibitem{AA} D.~Alekseevsky and A.~Arvanitoyeorgos, ``Riemannian flag manifolds with homogeneous geodesics,'' \emph{Trans.\ Amer.\ Math.\ Soc.}, \textbf{359} (2007), no.~8, pp.~3769--3789.

\bibitem{AC} A.~Ashikhmin and A.~R.~Calderbank,  ``Grassmannian packings from operator Reed--Muller codes,'' \emph{IEEE Trans.\ Inform.\ Theory}, \textbf{56} (2003), no.~11, pp.~5689--5714.

\bibitem{bar} H.~Bagherinia and R.~Manduchi, ``A theory of color barcodes,'' \emph{Proc.\ IEEE Int.\ Conf.\ Comput.\ Vis.} (ICCV), \textbf{14} (2011), pp.~806--813.

\bibitem{BN} A.~Barg and D.~Yu.~Nogin, ``Bounds on Packings of Spheres in the Grassmannian Manifold,'' \emph{IEEE Trans.\ Inform.\ Theory}, \textbf{48} (2002), no.~9, pp.~2450--2454.

\bibitem{ref2a} R.~Basri, T.~Hassner, and L.~Zelnik-Manor, ``Approximate nearest subspace search,'' \emph{IEEE Trans.\ Pattern Anal.\ Mach.\ Intell.}, \textbf{33} (2011), no.~2, pp.~266--278.

\bibitem{BES} C.~A.~Beattie, M.~Embree, and D.~C.~Sorensen, ``Convergence of polynomial restart Krylov methods for eigenvalue computations,'' \emph{SIAM Rev.}, \textbf{47} (2005), no.~3, pp.~492--515.

%\bibitem{Bhatia} R.~Bhatia, ``On the exponential metric increasing property,'' \emph{Linear Algebra Appl.}, \textbf{375} (2003), pp.~211--220.

%\bibitem{BH} R.~Bhatia and J.~Holbrook, ``Riemannian geometry and matrix geometric means,'' \emph{Linear Algebra Appl.}, \textbf{413} (2006), no.~2--3, pp.~594--618.

\bibitem{BG} \AA.~Bj\"{o}rck and G.~H.~Golub, ``Numerical methods for computing angles between linear subspaces,'' \emph{Math.\ Comp.}, \textbf{27} (1973), no.~123, pp.~579--594. 

%\bibitem{BCSS} L.~Blum, F.~Cucker, M.~Shub, and S.~Smale, \emph{Complexity and Real Computation}, Springer-Verlag, New York, NY, 1998.

%\bibitem{BSS} L.~Blum, M.~Shub, and S.~Smale, ``On a theory of computation and complexity over the real numbers,'' \emph{Bull.\ Amer.\ Math.\ Soc.}, \textbf{21} (1989), no.~1, pp.~1--46.

\bibitem{BNR} L.~Balzano, R.~Nowak, and B.~Recht, ``Online identification and tracking of subspaces from highly incomplete information,'' \emph{Annual Allerton Conf.\ Commun. Control Comput.}, \textbf{48} (2010), pp.~704--711.

%\bibitem{CK} G.~S.~Chirikjian and A.~B.~Kyatkin, \emph{Engineering Applications of Noncommutative Harmonic Analysis. With emphasis on rotation and motion groups}, CRC Press, Boca Raton, FL, 2001.

\bibitem{C} S.-C.~T.~Choi, ``Minimal residual methods for complex symmetric, skew symmetric, and skew Hermitian systems,'' \emph{preprint}, Report ANL/MCS-P3028-0812, Computation Institute, University of Chicago, IL, 2013. 

\bibitem{CHS} J.~H.~Conway, R.~H.~Hardin, and N.~J.~A.~Sloane, ``Packing lines, planes, etc.: packings in Grassmannian spaces,'' \emph{Experiment.\ Math.}, \textbf{5} (1996), no.~2, pp.~83--159.

\bibitem{motion} N.~P.~Da Silva and J.~P.~Costeira, ``The normalized subspace inclusion: Robust clustering of motion subspaces,'' \emph{Proc.\ IEEE Int.\ Conf.\ Comput.\ Vis.} (ICCV), \textbf{12} (2009), pp.~1444--1450.

\bibitem{DD} M.~M.~Deza and E.~Deza, \emph{Encyclopedia of Distances}, 2nd Ed., Springer, Heidelberg, 2013.

\bibitem{DHST} I.~S.~Dhillon, R.~W.~Heath, Jr.,~T.~Strohmer, and J.~A.~Tropp, ``Constructing packings in Grassmannian manifolds via alternating projection,'' \emph{Experiment.\ Math.}, \textbf{17} (2008), no.~1, pp.~9--35.

%\bibitem{EHOST} J.~Draisma, E.~Horobet, G.~Ottaviani, B.~Sturmfels, and R.~Thomas, ``The Euclidean distance degree of an algebraic variety,'' \emph{Found.\ Comput.\ Math.}, to appear.

\bibitem{ref2b} B.~Draper, M.~Kirby, J.~Marks, T.~Marrinan, and C.~Peterson, ``A flag representation for finite collections of subspaces of mixed dimensions,'' \emph{Linear Algebra Appl.}, \textbf{451} (2014), pp.~15--32.

\bibitem{EAS} A.~Edelman, T.~Arias, and S.~T.~Smith, ``The geometry of algorithms with orthogonality constraints,'' \emph{SIAM J.\ Matrix Anal.\ Appl.}, \textbf{20} (1999), no.~2, pp.~303--353.

\bibitem{eeg} N.~Figueiredo, P.~Georgieva, E.~W.~Lang, I.~M.~Santos, A.~R.~Teixeira, and A.~M.~Tom\'{e}, ``SSA of biomedical signals: A linear invariant systems approach,'' \emph{Stat.\ Interface}, \textbf{3} (2010), no.~3, pp.~345--355.

\bibitem{FH} R.~Fioresi and C.~Hacon, ``On infinite-dimensional Grassmannians and their quantum deformations,'' \emph{Rend.\ Sem.\ Mat.\ Univ.\ Padova}, \textbf{111} (2004), pp.~1--24. 

\bibitem{GVL} G.~Golub and C.~Van~Loan, \emph{Matrix Computations}, 4th Ed., John Hopkins University Press, Baltimore, MD, 2013.

%\bibitem{Gordon} C.~S.~Gordon, ``Homogeneous Riemannian manifolds whose geodesics are orbits,'' pp.~155--174, S.~Gindikin, Ed., \emph{Topics in Geometry}, Progress in Nonlinear Differential Equations and their Applications, \textbf{20}, Birkh\"auser, Boston, MA, 1996. 

%\bibitem{GH} P.~Griffiths and J.~Harris, \emph{Principles of Algebraic Geometry},  John Wiley, New York, NY, 1994.

\bibitem{HamLee2008} J.~Hamm and D.~D.~Lee, ``Grassmann discriminant analysis: A unifying view on subspace-based learning,'' \emph{Proc.\ Internat.\ Conf.\ Mach.\ Learn.} (ICML), \textbf{25} (2008), pp.~376--383.

\bibitem{HH} T.~F.~Hansen and D.~Houle, ``Measuring and comparing evolvability and constraint in multivariate characters,'' \emph{J.\ Evolution.\ Biol.}, \textbf{21} (2008), no.~5, pp.~1201--1219.

%\bibitem{Hatcher} A.~Hatcher, \emph{Algebraic Topology}, Cambridge University Press, Cambridge, 2002.

\bibitem{HRS} G.~Haro, G.~Randall, and G.~Sapiro, ``Stratification learning: Detecting mixed density and dimensionality in high dimensional point clouds,'' \emph{Proc.\ Adv.\ Neural Inform.\ Process.\ Syst.} (NIPS), \textbf{26} (2006), pp.~553--560.

\bibitem{mech} Q.~He, F.~Kong, and R.~Yan, ``Subspace-based gearbox condition monitoring by kernel principal component analysis, \emph{Mech.\ Systems Signal Process.}, \textbf{21} (2007), no.~4, pp.~1755--1772.

\bibitem{Helemskii} A.~Ya.~Helemskii, \emph{Lectures and Exercises on Functional Analysis},  Translations of Mathematical Monographs, \textbf{233}, AMS, Providence, RI, 2006.

%\bibitem{HL}C.~J.~Hillar and L.-H.~Lim, ``Most tensor problems are NP-hard,'' \emph{J.\ ACM}, \textbf{60} (2013), no.~6, Art.~45, 39~pp.

\bibitem{Horn} R.~A.~Horn and C.~R.~Johnson, \emph{Topics in Matrix Analysis}, Cambridge University Press, Cambridge, 1991.

%\bibitem{HHH} K.~H\"uper, U.~Helmke and S.~Herzberg, ``On the computation of means on Grassmann manifolds,'' \emph{Proc.\ Int.\ Symp.\ Math.\ Theory Networks Syst.} (MTNS), \textbf{19} (2010), pp.~2439--2441.

\bibitem{Husemoller} D.~Husemoller, \emph{Fibre Bundles}, 3rd Ed., Graduate Texts in  Mathematics, \textbf{20}, Springer, New York, NY, 1994.

%\bibitem{Jost} J.~Jost, \emph{Riemannian Geometry and Geometric Analysis}, 6th Ed., Springer, Heidelberg, 2011.

%\bibitem{Karcher}  H.~Karcher, ``Riemannian center of mass and mollifier smoothing,''  \emph{Comm.\ Pure Appl.\ Math.}, \textbf{30} (1977), no.~5, pp.~509--541.

\bibitem{Kato} T.~Kato, \emph{Perturbation Theory for Linear Operators}, Classics in Mathematics, Springer-Verlag, Berlin, 1995.

%\bibitem{KR} D.~A.~Klain and G.-C.~Rota, \emph{Introduction to Geometric Probability}, Lezioni Lincee, Cambridge University Press, Cambridge, 1997.

\bibitem{Knyazev} A.~V.~Knyazev and M.~E.~Argentati, ``Majorization for changes in angles between subspaces, Ritz values, and graph Laplacian spectra,'' \emph{SIAM  J.\ Matrix Anal.\ Appl}, \textbf{29} (2006), no.~1, pp.~15--32.

\bibitem{LZ} G.~Lerman and T.~Zhang, ``Robust recovery of multiple subspaces by geometric $l_p$ minimization,'' \emph{Ann.\ Statist.}, \textbf{39} (2011), no.~5, pp.~2686--2715.

\bibitem{blog} H.~Li and A.~Li, ``Utilizing improved Bayesian algorithm to identify blog comment spam,'' \emph{Proc.\ IEEE Symp.\ Robot.\ Appl.} (ISRA), \textbf{1} (2012), pp.~423--426.

\bibitem{LS} J.~Liesen and Z.~Strako\v{s}, \emph{Krylov Subspace Methods}, Oxford University Press, Oxford, 2013.

\bibitem{LWY} L.-H.~Lim, K.~S.-W.~Wong, and K.~Ye, ``Statistical estimation and the grassmannian of affine subspaces,'' \emph{prepirint}, (2016), \url{http://www.stat.uchicago.edu/~lekheng/work/affine.pdf}.

\bibitem{LHS} D.~J.~Love, R.~W.~Heath, Jr., and T.~Strohmer, ``Grassmannian beamforming for multiple-input multiple-output wireless systems,'' \emph{IEEE Trans.\ Inform.\ Theory}, \textbf{49} (2003), no.~10, pp.~2735--2747.

\bibitem{ref2c} D.~Luo and H.~Huang, ``Video motion segmentation using new adaptive manifold denoising model,'' pp.~65--72, \emph{Proc.\ IEEE Conf.\ Computer Vis.\ Pattern Recognit.} (CVPR), Columbus, OH, 2014.

\bibitem{MYDF} Y.~Ma, A.~Yang, H.~Derksen, and R.~Fossum, ``Estimation of subspace arrangements with applications in modeling and segmenting mixed data,'' \emph{SIAM Rev.}, \textbf{50} (2008), no.~3, pp.~413--458.
%\bibitem{MS} A.~Machado and I.~Salavessa, ``Grassmannian manifolds as subsets of Euclidean spaces,'' pp.~85--102,  L.~A.~Cordero, Ed., \emph{Differential Geometry}, Pitman Research Notes in Mathematics, \textbf{131}, Pitman, Boston, MA, 1985.

%\bibitem{Mattila} P.~Mattila, \emph{Geometry of Sets and Measures in Euclidean Spaces}, Cambridge Studies in Advanced Mathematics, \textbf{44}, Cambridge University Press, Cambridge, 1995.

\bibitem{MS} J.~W.~Milnor and J.~D. Stasheff, \emph{Characteristic Classes},  Annals of Mathematics Studies, \textbf{76}, Princeton University Press, Princeton, NJ, 1974.

\bibitem{Nicolaescu} L.~I.~Nicolaescu, \emph{Lectures on the Geometry of Manifolds}, 2nd Ed., World Scientific, Hackensack, NJ, 2007. 

%\bibitem{OSS} G.~Ottaviani, P.-J.~Spaenlehauer, and B.~Sturmfels, ``Exact solutions in structured low-rank approximation,'' \emph{preprint}, (2013), \url{http://arxiv.org/abs/1311.2376}.

\bibitem{econ} M.~Peng, D.~Bu, and Y.~Wang, ``The measure of income mobility in vector space,'' \emph{Physics Procedia}, \textbf{3} (2010), no.~5, pp.~1725--1732.

\bibitem{network} E.~Sharafuddin, N.~Jiang, Y.~Jin, and Z.-L.~Zhang, ``Know your enemy, know yourself: Block-level network behavior profiling and tracking,'' \emph{Proc.\ IEEE Global Telecomm.\ Conf.} (GLOBECOM), \textbf{53} (2010), pp.~1--6.

\bibitem{SLLM} B.~St.~Thomas, L.~Lin, L.-H.~Lim, and S.~Mukherjee, ``Learning subspaces of different dimensions,'' \emph{preprint}, (2014), \url{http://arxiv.org/abs/1404.6841}.

\bibitem{SS} G.~W.~Stewart and J.~Sun, \emph{Matrix Perturbation Theory}, Academic Press, Boston, MA, 1990.

\bibitem{SH} T.~Strohmer and R.~W.~Heath, Jr., ``Grassmannian frames with applications to coding and communication,'' \emph{Appl.\ Comput.\ Harmon.\ Anal.}, \textbf{14} (2003), no.~3, pp.~257--275.

\bibitem{SWF} X.~Sun, L.~Wang, and J.~Feng, ``Further results on the subspace distance,'' \emph{Pattern Recognition}, \textbf{40} (2007), no.~1, pp.~328--329.

\bibitem{VidMaSas2005} R.~Vidal, Y.~Ma, and S.~Sastry, ``Generalized principal component analysis,'' \emph{IEEE Trans.\ Pattern Anal.\ Mach.\ Intell.}, \textbf{27} (2005), no.~12, pp.~1945--1959.

%\bibitem{vN} J.~von~Neumann, ``Some matrix-inequalities and metrization of matric-space,'' \emph{Mitt.\ Forsch.-Inst.\ Math.\ Mech.\ Univ.\ Tomsk}, \textbf{1} (1937), pp.~286--299.

\bibitem{WWF} L.~Wang, X.~Wang, and J.~Feng, ``Subspace distance analysis with application to adaptive Bayesian algorithm for face recognition,'' \emph{Pattern Recognition}, \textbf{39} (2006), no.~3, pp.~456--464.

\bibitem{face} R.~Wang, S.~Shan, X.~Chen, and W.~Gao, ``Manifold-manifold distance with application to face recognition based on image set,'' \emph{Proc.\ IEEE Conf.\ Comput.\ Vis.\ Pattern Recognit.} (CVPR), \textbf{26} (2008), pp.~1--8.

\bibitem{Wong} Y.-C.~Wong, ``Differential geometry of Grassmann manifolds,'' \emph{Proc.\ Nat.\ Acad.\ Sci.}, \textbf{57} (1967), no.~3, pp.~589--594.

\bibitem{ref2d} J.~Yan and M.~Pollefeys, ``A general framework for motion segmentation: Independent, articulated, rigid, non-rigid, degenerate and
non-degenerate,'' pp.~94--106, \emph{Proc.\ European Conf.\ Computer Vis.} (ECCV), Graz, Austria, 2006.

\bibitem{ZT} L.~Zheng and D.~N.~C.~Tse, ``Communication on the Grassmann manifold: A geometric approach to the noncoherent multiple-antenna channel,'' \emph{IEEE Trans.\ Inform.\ Theory}, \textbf{48} (2002), no.~2, pp.~359--383.

\bibitem{text} G.~Zuccon, L.~A.~Azzopardi, C.~J.~van~Rijsbergen, ``Semantic spaces: Measuring the distance between different subspaces,''  pp.~225--236, P.~Bruza, D.~Sofge, W.~Lawless, K.~van~Rijsbergen, M.~Klusch (Eds), \emph{Quantum Interaction}, Lecture Notes in Artificial Intelligence, \textbf{5494}, Springer-Verlag, Berlin, 2009.
\end{thebibliography}

\end{document}